\documentclass{amsproc}
\usepackage{cite}
\usepackage{mathtools}
\usepackage{amssymb}
\usepackage{mathrsfs}
\usepackage{tikz}
\usetikzlibrary{cd, arrows, fit}

\usepackage[centertableaux]{ytableau}

\usepackage{enumitem}

\usepackage{bold}

\catcode `\@=11

\newtheorem{thm}{Theorem}[section]

\newtheorem{prop}[thm]{Proposition}

\newtheorem{cor}[thm]{Corollary}

\newtheorem{dfprop}[thm]{Definition-Proposition}

\theoremstyle{definition}

\newtheorem{definition}[thm]{Definition}
\newtheorem{rmk}[thm]{Remark}
\newtheorem{remark}[thm]{Remark}
\newtheorem{nota}[thm]{Notation}
\newtheorem{ex}[thm]{Example}
\newtheorem{example}[thm]{Example}

\def\O{\mathcal{O}}

\def\Vect{\mathcal{V}ect}

\def\Set{\mathcal{S}et}
\def\V{\mathcal V}

\makeatother
\newcommand{\EZDIAG}[5]{\xymatrix
@C+=2.5cm{*+[r]{#1}
\ar@(u,l)_(0.62){\displaystyle #5}[]
\ar@<1ex>^-{#3}[r]&\ar@<1ex>^-{#4}[l]#2}}

\def\del{\partial}
\def\id{{\mathrm{id}}}

\def\SS{\mathbb{S}}

\def\F{\mathcal{F}}
\def\FF{\mathfrak F}
\def\asts{{\mathcal V}}

\def\N{\mathbb{N}}

\def\Vect{\mathcal{V}ect}

\def\Graphs{{\mathcal G}\it raphs}

\def\Set{{\mathcal S}et}

\def\operads{{\mathfrak O}}

\def\F{\mathcal F}
\def\FF{\mathfrak F}

\def\N{{\mathbb N}}

\def\del{\partial}

\def\FinSet{\mathcal{F}in\mathcal{S}et}

\def\Surj{\mathsf{Surj}}

\def\O{{\mathcal O}}
\def\P{{\mathcal P}}
\def\SS{{\mathbb S}}

\newcommand{\Hom}{\operatorname{Hom}}
\newcommand{\Iso}{\operatorname{Iso}}

\def\V{\asts}
\def\asts{{\mathcal V}}
\def\F{\clusters}
\def\clusters{{\mathcal F}}

\def\ta{\twoheadrightarrow}

\def\ot{\otimes}

\newcommand{\Indec}{\mathcal{P}}
\def\Gpd{\mathcal{G}}
\def\M{\mathcal{M}}

\tikzstyle{none}=[]
\tikzstyle{White}=[fill=white, draw=black, shape=circle, inner sep=1.2mm]
\tikzstyle{SmallWhite}=[fill=white, draw=black, shape=circle, inner sep=0.6mm]
\tikzstyle{Black}=[fill=black, draw=black, shape=circle, inner sep=1.2mm]
\tikzstyle{SmallBlack}=[fill=black, draw=black, shape=circle, inner sep=0.6mm]
\tikzstyle{WhiteSq}=[fill=white, draw=black, shape=rectangle]
\tikzstyle{BlackSq}=[fill=black, draw=black, shape=rectangle]
\tikzstyle{GraySq}=[fill=gray, draw=black, shape=rectangle]
\tikzstyle{Red}=[fill=red, draw=black, shape=circle]
\tikzstyle{Blue}=[fill=blue, draw=black, shape=circle]

\tikzstyle{Arrow}=[->]
\tikzstyle{Mapsto}=[{|->}]
\tikzstyle{Dash}=[-, dashed]
\tikzstyle{DashArrow}=[->, dashed]
\tikzstyle{Double}=[<->]

\def\Corel{\mathcal{C}orel}
\def\Cospan{\mathcal{C}ospan}
\def\Span{\mathcal{S}pan}

\setlist[description]{font=\normalfont\scshape}

\title{Calculations for plus constructions}
\subjclass[2010]{Primary 18M60}

\author{Michael Monaco}
\email{monacom@purdue.edu}
\address{Purdue University Department of Mathematics, West Lafayette, IN 47907}

\begin{document}

\begin{abstract}
  In recent work by Kaufmann and the author, we define a plus construction for monoidal categories and showed that if the monoidal category is a unique factorization category, then the plus construction yields a Feynman category.
  In this paper, we will use different methods to construct UFCs and demonstrate how the plus construction reproduces and clarifies existing constructions through explicit computations.
\end{abstract}

\maketitle
\setcounter{tocdepth}{2}

\section*{Introduction}

In the framework of Feynman categories, the plus construction plays a fundamental role in describing algebras, twists, and more recently monoid definitions.
In \cite{KMoManin}, they extend the plus construction to a type of monoidal category which they call a unique factorization category.
We will use techniques from different areas to give explicit examples of these unique factorization categories and give applications of the plus construction.
This paper will focus on four fundamental examples: the trivial Feynman category, the category of finite sets, the category of cospans, and the category of spans.

As described by Kaufmann and Ward in \cite{feynman}, the plus constructions $(\FF^{triv})^+$ and $(\FinSet)^+$ are related to monoids and operads respectively.
We will add to this by considering the nc-plus construction introduced in \cite{KMoManin} and operads with multiplication.

The category of cospans has finite sets as objects and equivalence classes of diagrams $S \rightarrow V \leftarrow T$ as morphisms.
In algebraic topology, cospans have proven to be a useful way to work with cobordisms and related structures.
Examples of this sort can be found in the work of Grandis~\cite{grandis1,grandis2,grandis3}, Feshbach and Voronov~\cite{higher-cobordism-TQFT}, and Steinebrunner~\cite{steinebrunner}, just to name a few.
Following this line of thought, we will show how Frobenius algebras naturally arise from a combinatorial version of the plus construction.
We will also use the notion of a structured cospan due to Baez and Courser~\cite{baez2019structured, courser2020open} to give colored versions of properads.
The nc-plus construction of cospans and its relation to mergers will also be discussed.

In \cite{KMoManin}, they show that the category of spans is a hereditary unique factorization category.
This implies that $\Span^+$ is a Feynman category.
Despite having some interesting properties, the gadgets corepresented by $\Span^+$ have not been explored to the same extent as the gadgets corepresented by $\Cospan^+$.
To better understand the algebraic significance of this structure, we will briefly survey some of the combinatorial and categorical properties of $\Span$.
We will then end by describing the relation of $\Span$ and $\Span^+$ to bialgebras.

\subsection*{Acknowledgments}

We would like to thank Ralph Kaufmann for his continued support and thank Philip Hackney and Jan Steinebrunner for helpful discussions on this topic.
We also appreciate the work of the referee and the helpful feedback they gave.

\section{Plus constructions, Feynman categories, and UFCs}

In \cite{categories-seriously}, Lawvere advocates taking descriptions of mathematical objects as categories ``seriously''.
For example, a group $G$ canonically determines a category $\Sigma G$ with one object and $G$ as a set of morphisms.
Applying this attitude to $\Sigma G$, one recovers many classical ideas such as representations, intertwining operators, induced representations, and Frobenius reciprocity as special cases of different categorical constructions.
Because of this, it is common to blur the distinction between these concepts.
However, it is useful to keep this distinction to separate the datum describing the structure (the set $G$ with a binary operation and certain properties) from the structure itself (the category $\Sigma G$).

These sorts of distinctions are especially important in the study of operad-like structures.
In \cite{feynman}, Kaufmann and Ward introduced a special type of monoidal category called a Feynman category to encode different ``types'' of operad-like structure.
In their formalism, a strong monoidal functor $\O: \F \to \mathcal{C}$  encodes an operad-like structure of ``type $\F$'' in a category $\mathcal{C}$.
For Feynman categories that come from a plus construction, a strong monoidal functor $\O: \F^+ \to \mathcal{C}$ determines a category $\F_{\O}$ by a so-called indexed enrichment.
This process of going from the datum $\O$ to the category $\F_\O$ plays a key role in describing algebras and the theory of twists.
In \cite{KMoManin}, they extend this by defining the plus construction for a larger class of monoidal categories and show that $\M^+$ is a Feynman category when $\M$ is a unique factorization category.
In this section, we will briefly recall the definitions of unique factorization categories and the plus construction.

\subsection{Unique factorization categories}

\begin{nota}
  Let $\mathcal{C}^\boxtimes$ denote the free symmetric monoidal category generated by the category $\mathcal{C}$.
\end{nota}

\begin{definition}
  A \emph{unique factorization category} (UFC) is a triple $(\M, \Indec, \jmath)$ where $\M$ is a symmetric monoidal category, $\Indec$ is a groupoid, and $\jmath$ is a functor $\jmath: \Indec \to \Iso(\M\downarrow \M)$ such that:
  \begin{enumerate}
  \item The functor $\jmath^{\boxtimes}: \Indec^{\boxtimes} \to \Iso(\M \downarrow \M)$ induced from the monoidal structure of $\M$ induces an equivalence of categories.
  \item The slice categories of $\M$ are essentially small.
  \end{enumerate}
  We will often refer to a unique factorization category by its underlying monoidal category $\M$.
  The pair $(\Indec, \jmath)$ will be called a {\em basis of morphisms} and its elements will be called {\em irreducibles} or {\em basic morphisms}.
\end{definition}

\begin{definition}
  \cite{KMoManin}
  A \emph{Feynman category} $\F$ is a unique factorization category with additional data $(\V, \imath)$ such that: 
  \begin{enumerate}
  \item The functor $\imath^\boxtimes: \V^\boxtimes \stackrel{\sim}{\to} \Iso(\M)$ induces an equivalence.
  \item The irreducible morphisms are given by $\P = \Iso(\F\downarrow \V)$ with $\jmath = (id_{\F},id_{\F},\imath)$ as a compatible choice of basic morphisms $\P$ making $\F$ into a unique factorization category.
  \end{enumerate}
  A choice of such a pair $(\V,\imath)$ will be called a {\em basis of objects} and its elements will be called {\em irreducibles} or {\em basic objects}.
\end{definition}

\begin{dfprop}
  \cite{KMoManin}
  A basis for morphisms $(\Indec,\jmath)$ is \emph{hereditary} if for every pair of composable morphisms $(\phi_0,\phi_1)$, with $\phi_1 \circ \phi_0=\phi$, and decomposition into irreducible morphisms 
  \begin{equation}
    \phi_0\simeq \bigotimes_{v\in V}\phi_{0,v}, \quad  \phi_1\simeq \bigotimes_{w\in W} \phi_{1,w},  \text{ and }\phi=\bigotimes_{u\in U}\phi_u
  \end{equation}    
  there exists a partition of $V \amalg W = \amalg_{u\in U} P_u$ indexed by $U$, such that for each $u \in U$ there is a decomposition pair $(\phi_{0,u},\phi_{1,u})$ of the $\phi_u$,  viz.\ $\phi_{1,u} \phi_{0,u}=\phi_u$, such that
  \begin{equation}
    \label{eq:hereditarycond}
    \phi_{0,u} \simeq \bigotimes_{v\in P_u\cap V}\phi_{0,v} \text{ and } \phi_{1,u} \simeq \bigotimes_{w\in P_u\cap W}\phi_{1,w}
  \end{equation}
  A unique factorization category is a {\em hereditary UFC} if its basis is hereditary.
\end{dfprop}

\begin{rmk}
  Unique factorization categories are versatile structures which admit many descriptions.
  In \cite{KMoManin}, these conditions were equivalently formulated as right Ore conditions.
  In \cite[Proposition 6.31]{KMoManin}, they also show that a hereditary UFC $\M$ naturally determines an indexing functor $\M \to \Cospan$.
  Categories equipped with indexing functors that satisfy the appropriate conditions play an important role in the work of Steinebrunner~\cite{steinebrunner} where they go by the name \emph{labeled cospan categories}.
  Moreover, Hackney and Beardsley  \cite{labeled-cospan-properad} describe a connection between these labeled cospan categories and Segal presheaves of a category of level graphs.
\end{rmk}

\subsection{Plus constructions of categories}

\begin{definition}
  \cite{KMoManin}
  Given a category $\mathcal{C}$, define $\mathcal{C}^{\boxtimes +}$ so that
  \begin{enumerate}
  \item The objects are words $\phi_1 \boxtimes \ldots \boxtimes \phi_n$ of morphisms $\phi_i \in \mathcal{C}$.
  \item The morphisms are generated by two types of basic morphisms:
    \begin{description}
    \item[Isomorpisms] are words $(\sigma_1\Downarrow\sigma'_1) \boxtimes \ldots \boxtimes (\sigma_n\Downarrow\sigma'_n)$ where the $\sigma_i$ and $\sigma_i'$ are isomorphisms of $\mathcal{C}$.
    \item[$\gamma$-morphisms] for every composable pair  $(\phi_1,\phi_0)$ there is a generator
      \begin{equation*}
        \gamma_{\phi_1,\phi_0}: \phi_1 \boxtimes \phi_0 \to \phi_1 \circ \phi_0
      \end{equation*}
    \end{description}
  \item There are several relations including the typical ones like associativity, identities, and interchange as well as some new ones like equivariance with respect to isomorphisms.
    See \cite{KMoManin} for details.
  \end{enumerate}
\end{definition}

\begin{definition}
  \cite{KMoManin}
  In the case where $\mathcal{C} = \M$ is a monoidal category, there is a refinement:
  \begin{enumerate}
  \item The \emph{nc-plus construction} $\M^{nc +}$ is obtained from $\M^{\boxtimes +}$ by adjoining new generators $\mu_{\phi_0,\phi_1}: \phi_1 \boxtimes \phi_2 \to \phi_1 \ot \phi_2$ and imposing new relations.
  \item The \emph{(localized) plus construction} $\M^{+loc}$ is defined to be the localization of $\mu$.
    That is, we add the morphism $\mu^{-1}$ and mod out by $\mu \circ \mu^{-1} = \mu^{-1} \circ \mu = \id$.
  \end{enumerate}
  In principle, these localizations can be difficult to compute.
  However when $\M$ is a hereditary unique factorization category, \cite{KMoManin} defines a plus construction $\M^{+}$ which is a monoidally equivalent to $\M^{+loc}$.
  Moreover, the hereditary property is characterized as a right Ore condition, so there is a right roof calculus making the computations tractable.
  Since we only work with hereditary UFCs, we will not make a distinction between $\M^{+loc}$ and $\M^{+}$ and we will refer to both of them as ``the plus construction''.
\end{definition}

\begin{prop}\cite{KMoManin}
  The plus construction of a hereditary unique factorization category is a Feynman category.
\end{prop}

\section{Trivial Feynman category}

Define the \emph{trivial category} $1$ to be the category with a single object $\ast$ and a single identity morphism.
Define the \emph{trivial Feynman category} $\FF^{triv}$ so that $\F = 1^{\boxtimes}$ and $\V = 1$.
Explicitly, the objects of $\FF^{triv}$ are strings $\ast^{\boxtimes n}$ with the empty string as the unit and the morphisms are the commutativity constraints.

\subsection{Monoids}

For a Feynman category $\FF = (\F, \V, \imath)$, a strong monoidal functor $\O: \F \to \mathcal{C}$ is called an $\F$-op in $\mathcal{C}$.
The name is supposed to evoke the idea that $\O$ is an operad-like structure of ``$\F$-type''.
In our case, an op of the trivial Feynman category $\FF^{triv}$ is essentially a choice of an object in $\mathcal{C}$.
Interestingly, a lot of theory can be ``boot-strapped'' from this simple example.
For a full explanation of this idea, we refer the reader to Section 3.4 of \cite{feynmanrep}.
For our purposes, we single-out the following fact:

\begin{prop}[Kaufmann~\cite{feynmanrep}]
  \label{triv-plus-corep-monoids}
  As a combinatorial object, $(\FF^{triv})^+$ is equivalent to the monoidal category $\mathcal{S}urj^<$ whose objects are finite sets and whose morphisms are surjections with ordered fibers.
  Moreover $(\FF^{triv})^+$ corepresents monoids as a Feynman category.
\end{prop}

\begin{example}
  To understand $(\FF^{triv})^+$ as a combinatorial category, it is helpful to think of it in terms of the diagrams of \cite{feynmanrep}.
  In their depiction, an object of $(\FF^{triv})^+$ is represented as a string consisting of the basic object $\id_{\ast}$ and a morphism is represented as a way of stacking these objects.
  
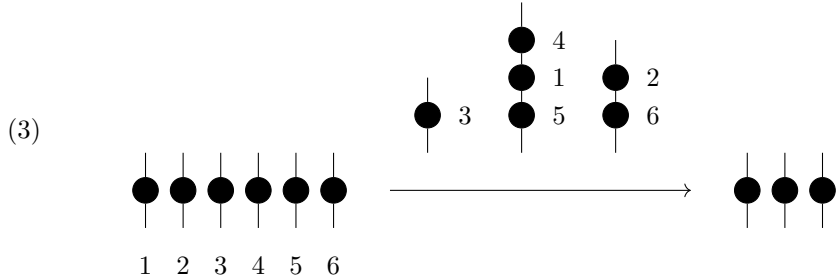
\begin{equation}
    \begin{tikzpicture}[baseline=(current  bounding  box.center)]
      \node [style=Black] (0) at (-3.75, -1.25) {};
      \node [style=none] (1) at (-3.75, -0.75) {};
      \node [style=none] (2) at (-3.75, -1.75) {};
      \node [style=Black] (3) at (-3.25, -1.25) {};
      \node [style=none] (4) at (-3.25, -0.75) {};
      \node [style=none] (5) at (-3.25, -1.75) {};
      \node [style=Black] (6) at (-2.75, -1.25) {};
      \node [style=none] (7) at (-2.75, -0.75) {};
      \node [style=none] (8) at (-2.75, -1.75) {};
      \node [style=Black] (9) at (-2.25, -1.25) {};
      \node [style=none] (10) at (-2.25, -0.75) {};
      \node [style=none] (11) at (-2.25, -1.75) {};
      \node [style=Black] (12) at (-1.75, -1.25) {};
      \node [style=none] (13) at (-1.75, -0.75) {};
      \node [style=none] (14) at (-1.75, -1.75) {};
      \node [style=Black] (15) at (-1.25, -1.25) {};
      \node [style=none] (16) at (-1.25, -0.75) {};
      \node [style=none] (17) at (-1.25, -1.75) {};
      \node [style=none] (18) at (-3.75, -2.25) {$1$};
      \node [style=none] (19) at (-3.25, -2.25) {$2$};
      \node [style=none] (20) at (-2.75, -2.25) {$3$};
      \node [style=none] (21) at (-2.25, -2.25) {$4$};
      \node [style=none] (22) at (-1.75, -2.25) {$5$};
      \node [style=none] (23) at (-1.25, -2.25) {$6$};
      \node [style=Black] (24) at (1.25, 0.25) {};
      \node [style=Black] (27) at (2.5, 0.25) {};
      \node [style=none] (28) at (2.5, 0.75) {};
      \node [style=Black] (30) at (0, -0.25) {};
      \node [style=none] (31) at (0, 0.25) {};
      \node [style=none] (32) at (0, -0.75) {};
      \node [style=Black] (33) at (1.25, 0.75) {};
      \node [style=none] (34) at (1.25, 1.25) {};
      \node [style=Black] (36) at (1.25, -0.25) {};
      \node [style=none] (38) at (1.25, -0.75) {};
      \node [style=Black] (39) at (2.5, -0.25) {};
      \node [style=none] (41) at (2.5, -0.75) {};
      \node [style=none] (42) at (1.75, 0.25) {$1$};
      \node [style=none] (43) at (3, 0.25) {$2$};
      \node [style=none] (44) at (0.5, -0.25) {$3$};
      \node [style=none] (45) at (1.75, 0.75) {$4$};
      \node [style=none] (46) at (1.75, -0.25) {$5$};
      \node [style=none] (47) at (3, -0.25) {$6$};
      \node [style=none] (48) at (-0.5, -1.25) {};
      \node [style=none] (49) at (3.5, -1.25) {};
      \node [style=Black] (50) at (4.25, -1.25) {};
      \node [style=none] (51) at (4.25, -0.75) {};
      \node [style=none] (52) at (4.25, -1.75) {};
      \node [style=Black] (53) at (4.75, -1.25) {};
      \node [style=none] (54) at (4.75, -0.75) {};
      \node [style=none] (55) at (4.75, -1.75) {};
      \node [style=Black] (56) at (5.25, -1.25) {};
      \node [style=none] (57) at (5.25, -0.75) {};
      \node [style=none] (58) at (5.25, -1.75) {};
      \draw (1.center) to (0);
      \draw (0) to (2.center);
      \draw (4.center) to (3);
      \draw (3) to (5.center);
      \draw (7.center) to (6);
      \draw (6) to (8.center);
      \draw (10.center) to (9);
      \draw (9) to (11.center);
      \draw (13.center) to (12);
      \draw (12) to (14.center);
      \draw (16.center) to (15);
      \draw (15) to (17.center);
      \draw (28.center) to (27);
      \draw (31.center) to (30);
      \draw (30) to (32.center);
      \draw (34.center) to (33);
      \draw (36) to (38.center);
      \draw (39) to (41.center);
      \draw (33) to (24);
      \draw (24) to (36);
      \draw (27) to (39);
      \draw [style=Arrow] (48.center) to (49.center);
      \draw (51.center) to (50);
      \draw (50) to (52.center);
      \draw (54.center) to (53);
      \draw (53) to (55.center);
      \draw (57.center) to (56);
      \draw (56) to (58.center);
    \end{tikzpicture}
  \end{equation}
\end{example}

\subsection{Graded monoids}

If we acknowledge the importance of the trivial Feynman category for the plus construction, then it is natural to consider the nc-plus construction of this category as well.
First, observe that the category $(\FF^{triv})^{nc+}$ has a basic object $\id_{\ast^{\boxtimes n}}$ for each natural number $n$, hence any $(\FF^{triv})^{nc+}$-op will naturally involve a grading.

To fix notation, we denote a graded object by $A = \{A_i\}_{i \in \N}$.
When finite coproducts exist, there is a canonical tensor product of graded objects in $\mathcal{C}$:
\begin{equation}
  (A \bullet B)_n = \coprod_{n=p+q} A_p \otimes B_q
\end{equation}
The monoidal unit for this product is then
\begin{equation}
 (1_{Gr})_n =
  \begin{cases}
    1_{\mathcal{C}}, & n=0 \\ 0_{\mathcal{C}}, & n > 0
  \end{cases}
\end{equation}
Note that a pointing $\eta: 1_{Gr} \to A$ in graded $\mathcal{C}$-objects amounts to a pointing $1_{\mathcal{C}} \to A_0$ in $\mathcal{C}$.
Hence the unit conditions for $\eta$ say that the following is an isomorphism:
\begin{equation*}
  A_n \to 1_{\mathcal{C}} \otimes A_n \to A_0 \otimes A_n \to A_n
\end{equation*}

\begin{prop}
  The category $(\F^{triv})^{nc+}$ corepresents monoid objects in the category of symmetric sequences of $\mathcal{C}$-monoids.
\end{prop}
\begin{proof}
  Fix a strong monoidal functor $\O: (\F^{triv})^{nc+} \to \mathcal{C}$.
  Define an $\SS_n$-module $M_{\O}(n) = \O(\id_{\ast^{\boxtimes n}})$ with the action $\SS_n \to \mathrm{Aut}(M_{\O})$ coming from commutativity constraint.
  All together, this forms an $\SS$-module $M_{\O} = \{M_{\O}(n)\}_{n \in \N}$.
  The image of $\O$ on the morphism $\gamma: \id_{\ast^{\boxtimes n}} \boxtimes \id_{\ast^{\boxtimes n}} \to \id_{\ast^{\boxtimes n}}$ is a morphism $M_{\O}(n) \otimes M_{\O}(n) \to M_{\O}(n)$ which makes each $M_{\O}(n)$ into a monoid.
  The image of $\O$ on the $\mu$-morphism $\mu: \id_{\ast^{\boxtimes n}} \boxtimes \id_{\ast^{\boxtimes m}} \to \id_{\ast^{\boxtimes (n+m)}}$ is a morphism $M_{\O}(n) \otimes M_{\O}(m) \to M_{\O}(n+m)$.
  These assemble into a morphism $M_{\O} \bullet M_{\O} \to M_{\O}$.
  Hence we know that $M_{\O} = \{M_{\O}(n)\}_{n \in \N}$ is at least a monoid object in the category of $\N$-graded sets.
  
  The interchange relation for the plus construction implies that the following diagram commutes:
  \begin{equation}
    \begin{tikzcd}
      M_{\O}(n) \otimes M_{\O}(m) \otimes M_{\O}(n) \otimes M_{\O}(m) \arrow[d] \arrow[r] & M_{\O}(n+m) \otimes M_{\O}(n+m) \arrow[d] \\
      M_{\O}(n) \otimes M_{\O}(m) \arrow[r] & M_{\O}(n+m)
    \end{tikzcd}
  \end{equation}
  Therefore $M_{\O} \bullet M_{\O} \to M_{\O}$ respects the monoid structure of $M_{\O}$ making it into a monoid objects in the category of symmetric sequences of $\mathcal{C}$-monoids.
\end{proof}

\begin{cor}
  Any monoid (an $(\FF^{triv})^+$-op) pulls-back to a monoid object in the category of symmetric monoids (an $(\FF^{triv})^{nc+}$-op).
\end{cor}
\begin{proof}
  Abstractly, this is just a pullback of the quotient functor $(\FF^{triv})^{nc+} \to (\FF^{triv})^+$.
  Concretely, given a monoid $M$, we define the graded object $A = \{A_i\}_{i \in \N}$ by $A_i = M^{\otimes i}$.
  Define $\mu: A_p \otimes A_q \to A_{p+q}$ to be concatenation.
  To define $\gamma: A_n \otimes A_n \to A_n$, write the multiplication of $M$ as $m: M \otimes M \to M$ and let $C$ be the commutativity constraint associated to the following permutation:
  \begin{equation}
    \begin{pmatrix}
      1 & 1+n & \ldots & n & n+n \\
      1 & 2 & \ldots & 2n-1 & 2n
    \end{pmatrix}
  \end{equation}
  Then $\gamma: A_n \otimes A_n \to A_n$ is the following composition:
  \begin{equation*}
    A_n \otimes A_n
    \overset{C}{\to}
    (A \otimes A)^{\otimes n}
    \overset{m^{\otimes n}}{\to}
    A^{\otimes n} = A_n
    \qedhere
  \end{equation*}
\end{proof}

\section{Finite Sets}

The category $\FinSet$ of finite sets is a Feynman category where the singleton sets are the basic objects and any map $f: X \to Y$ can be factored into a collection of maps $\{f^{-1}(y) \to \{y\}\}_{y \in Y}$.
We will see a connection to operads and make a new connection to operads with multiplication.
We will also point out a structural similarity to Young tableaux which we think is interesting.

\subsection{Operads}
\label{sec:operads}

We can think of the plus construction as ascending upwards in some algebraic hierarchy.
For example, we have seen that ``above'' objects ($\FF^{triv}$-ops), there are monoids ($(\FF^{triv})^+$-ops).
Similarly, ``above'' commutative monoids ($\FinSet$-ops), there are operads ($\FinSet^+$-ops).

\begin{prop}[Kaufmann~\cite{feynmanrep}]
  \label{FinSet-corep-com-monoids}
  The category $\FinSet$ of finite sets corepresents commutative monoids.
\end{prop}

\begin{prop}[Kaufmann~\cite{feynmanrep}]
  \label{FinSet-plus-corep-operads}
  Combinatorially, $\FinSet^+$ is equivalent to a Borisov--Manin category of graphs with rooted corollas as objects and the generating morphisms have level trees as ghost graphs.
  Moreover, they corepresent operads as Feynman categories.
\end{prop}

\begin{example}
  Let $\phi_n$ denote some $n$-to-$1$ map in $\FinSet$.
  Then the morphisms $\gamma_{\phi_2, \phi_1 \amalg \phi_3} \to \phi_4$ in $\FinSet^+$ can be identified with a morphism in a Borisov--Manin category:
  \begin{equation}
    \begin{tikzpicture}[baseline=(current  bounding  box.center)]
      \node [style=White] (0) at (-4.5, 0) {};
      \node [style=White] (1) at (-3.5, 0) {};
      \node [style=White] (2) at (-2.25, 0) {};
      \node [style=none] (3) at (-4.75, 0.5) {};
      \node [style=none] (4) at (-4.25, 0.5) {};
      \node [style=none] (5) at (-3.5, 0.5) {};
      \node [style=none] (6) at (-2.75, 0.5) {};
      \node [style=none] (7) at (-2.25, 0.5) {};
      \node [style=none] (8) at (-1.75, 0.5) {};
      \node [style=White, inner sep = 0.6mm] (9) at (-0.5, 0.75) {};
      \node [style=White, inner sep = 0.6mm] (10) at (-0.75, 1) {};
      \node [style=White, inner sep = 0.6mm] (11) at (-0.25, 1) {};
      \node [style=none] (14) at (-0.75, 1.25) {};
      \node [style=none] (15) at (-0.5, 1.25) {};
      \node [style=none] (16) at (-0.25, 1.25) {};
      \node [style=none] (17) at (0, 1.25) {};
      \node [style=White] (18) at (1.5, 0) {};
      \node [style=none] (19) at (0.75, 0.5) {};
      \node [style=none] (20) at (1.25, 0.5) {};
      \node [style=none] (21) at (1.75, 0.5) {};
      \node [style=none] (22) at (2.25, 0.5) {};
      \node [style=none] (23) at (-1.25, 0) {};
      \node [style=none] (24) at (0.25, 0) {};
      \node [style=none] (25) at (-4.5, -0.5) {};
      \node [style=none] (26) at (-3.5, -0.5) {};
      \node [style=none] (27) at (-2.25, -0.5) {};
      \node [style=none] (28) at (1.5, -0.5) {};
      \node [style=none] (29) at (-0.5, 0.5) {};
      \draw (3.center) to (0);
      \draw (4.center) to (0);
      \draw (5.center) to (1);
      \draw (6.center) to (2);
      \draw (7.center) to (2);
      \draw (8.center) to (2);
      \draw (14.center) to (10);
      \draw (15.center) to (11);
      \draw (16.center) to (11);
      \draw (17.center) to (11);
      \draw (10) to (9);
      \draw (11) to (9);
      \draw (19.center) to (18);
      \draw (20.center) to (18);
      \draw (21.center) to (18);
      \draw (22.center) to (18);
      \draw [style=Arrow] (23.center) to (24.center);
      \draw (0) to (25.center);
      \draw (1) to (26.center);
      \draw (2) to (27.center);
      \draw (18) to (28.center);
      \draw (9) to (29.center);
    \end{tikzpicture}
  \end{equation}
  The image of $\gamma_{\phi_2, \phi_1 \amalg \phi_3} \to \phi_4$ under a strong monoidal functor $\O: \FinSet^+ \to \mathcal{C}$ is the same thing as an operadic composition $\O(2) \otimes \O(1) \otimes \O(3) \to \O(4)$.
\end{example}

\begin{prop}[Kaufmann~\cite{feynmanrep}]
  Combinatorially, $(\FinSet^{<})^+$ is equivalent to a decorated Borisov--Manin category of graphs with planar rooted corollas as objects and the generating morphisms have level trees as ghost graphs.
  Moreover, they corepresent non-symmetric operads as Feynman categories.
\end{prop}

\subsection{Operads with multiplication}
\label{sec:operad-mult}

We will show how to obtain a Feynman category which corepresents operads with multiplication by using a slight modification of the plus construction.
Rather than shift the focus from a particular example to a general construction, we will simply introduce this modification as an ad hoc construction.

\begin{prop}
  Define $\operads_{\pi}$ by starting with $\FinSet^+$ and allowing words of morphisms of the form $(\sigma_1 \Downarrow \pi_1) \boxtimes \ldots \boxtimes (\sigma_n\Downarrow \pi_n)$ where $\sigma_i$ are isomorphisms and $\pi_i$ are surjections.
  Then $\operads_\pi$ is a Feynman category that corepresents operads with multiplication.
\end{prop}
\begin{proof}
  The added actions $(\sigma \Downarrow \pi)$ can be factored as $(\sigma_1 \Downarrow t_1) \boxtimes \ldots \boxtimes (\sigma_n \Downarrow t_n)$ where $t_i$ are surjections with singletons as codomains.
  The source of $(\sigma_n \Downarrow t_n)$ is an aggregate of corollas and the target is a single corolla.
  Hence the new morphisms meet the necessary conditions, so $\operads_{\pi}$ is indeed a Feynman category.
  
  Let $\O: \operads_\pi \to \mathcal{C}$ be a strong monoidal functor.
  Since $\operads \simeq \FinSet^+$ is present as a subcategory, we still have the $\SS_n$-actions and operadic compositions:
  \begin{equation}
    \begin{tikzcd}
      \O(\phi_{n_1}) \otimes \ldots \otimes \O(\phi_{n_k}) \otimes \O(\phi_k) \arrow[r, "\gamma"] & \O\left(\phi_{\sum_i n_i} \right)
    \end{tikzcd}
  \end{equation}
  On the other hand, the new actions introduce a multiplication:
  \begin{equation}
    \begin{tikzcd}
      \O(\phi_n) \otimes \O(\phi_m) \arrow[r, "\sim"] & \O(\phi_n \amalg \phi_m) \arrow[r, "\O(Id \Downarrow \pi)"] & \O(\phi_{n+m})
    \end{tikzcd}  
  \end{equation}
  Therefore $\O$ is the same data as an operad with multiplication.
\end{proof}

\begin{cor}
  Pulling back along the inclusion $\iota: \FinSet^+ \to \operads_{\pi}$ forgets the multiplication structure.
  \qed
\end{cor}

\begin{example}
  Let $\mathsf{FinSet}$ denote the skeletal category with the sets $\underline{n} = \{1, \ldots, n\}$ as objects and functions as morphism.
  Then define $\operads_{\pi}^{skel}$ in a similar manner by starting with $\mathsf{FinSet}^+$ and allowing words of the form $(\sigma_1 \Downarrow \pi_1) \boxtimes \ldots \boxtimes (\sigma_n\Downarrow \pi_n)$.
  Fix an operad $\O: \mathsf{FinSet}^+ \to \Vect$ and write $\O(n) = \O(\underline{n} \to \underline{1})$ as short hand.
  Now define a strong monoidal functor $\O^{nc}: \operads_{\pi}^{skel} \to \Vect$ on basic objects by
  \begin{equation}
    \O^{nc}(n) = \bigoplus_{n = \sum_{i=1}^k n_i}
    \O(n_1) \otimes \ldots \otimes \O(n_k)
  \end{equation}
  \begin{enumerate}
  \item The operadic composition in $\O^{nc}$ is given by summing over all possible operadic compositions in $\O$.
  \item The multiplication is defined by commutativity of finite colimits and tensors followed by inclusion:
    \begin{equation}
      \begin{tikzcd}
        \left(
          \displaystyle \bigoplus_{n = \sum_{i=1}^N n_i}
          \O(n_1) \otimes \ldots \otimes \O(n_N)
        \right)
        \otimes
        \left(
          \displaystyle \bigoplus_{m = \sum_{j=1}^M m_j}
          \O(m_1) \otimes \ldots \otimes \O(m_M)
        \right)
        \arrow[d, "\sim"] \\
        \displaystyle \bigoplus_{n = \sum_{i=1}^N n_i}
        \displaystyle \bigoplus_{m = \sum_{j=1}^M m_j}
        \O(n_1) \otimes \ldots \otimes \O(n_N)
        \otimes
        \O(m_1) \otimes \ldots \otimes \O(m_M)
        \arrow[d, tail] \\
        \displaystyle \bigoplus_{m+n = \sum_{k=1}^P p_k}
        \O(p_1) \otimes \ldots \otimes \O(p_P)
      \end{tikzcd}  
    \end{equation}
  \end{enumerate}
\end{example}

\subsection{Young tableaux}

We can think of a surjection $p: E \ta B$ as a $B$-indexed partition of a set $E$ by taking fibers $\{ p^{-1}(b) : b \in B \}$.
In the representation theory of symmetric groups, partitions are encoded by Young tableaux/tabloids.
In this section, we will look at this structure from the point of view of the plus construction.

\subsubsection{Tabloids and Tableaux as nc-plus constructions}

\begin{definition}
  To match existing conventions, we will work with the following skeletal categories:
  \begin{enumerate}
  \item Let $\Surj$ be the subcategory of $\FinSet$ whose objects are sets $\underline n = \{1, \ldots, n\}$ and whose morphisms are surjections.
  \item Let $\Surj^{<}$ be the category with sets $\underline n = \{1, \ldots, n\}$ as the objects and surjections with an ordering on each fiber as the morphisms.
  \end{enumerate}
\end{definition}

\begin{example}[Young tabloids]
  Consider \eqref{eq:tabloid}, we will think of an object $f$ of $\Surj^{nc+}$ on the left as the diagram on the right where the rows are the fibers of $f$.
  Since the fibers are not ordered, the corresponding rows are not ordered either which gives us a Young tabloid.
  \begin{equation}
    \label{eq:tabloid}
    \begin{tikzpicture}[baseline={(current bounding box.center)}]
      \node [style=none] (0) at (1.75, 0.5) {$1$};
      \node [style=none] (1) at (-0.75, 0.5) {$2$};
      \node [style=none] (2) at (0.75, 0.5) {$3$};
      \node [style=none] (3) at (-1.75, 0.5) {$4$};
      \node [style=none] (4) at (1.25, 0.5) {$5$};
      \node [style=none] (5) at (-0.25, 0.5) {$6$};
      \node [style=none] (6) at (0.25, 0.5) {$7$};
      \node [style=none] (7) at (-1.25, 0.5) {$8$};
      \node [style=none] (8) at (-1.5, -0.5) {$1$};
      \node [style=none] (9) at (-0.25, -0.5) {$2$};
      \node [style=none] (10) at (1, -0.5) {$3$};
      \node [style=none] (11) at (1.75, -0.5) {$4$};
      \draw (3) to (8);
      \draw (7) to (8);
      \draw (5) to (9);
      \draw (6) to (9);
      \draw (1) to (9);
      \draw (2) to (10);
      \draw (0) to (11);
      \draw (4) to (10);
    \end{tikzpicture}
    \hspace{2cm}
    \ytableausetup{boxsize=normal, tabloids}
    \ytableaushort{
      48,672,53,1
    }
  \end{equation}

  Under this interpretation, precomposing by a surjection lengthens the rows and postcomposing by an injection adds rows.
  One place where adding columns and rows appears naturally is the branching rules for representations of the symmetric group.
\end{example}

\begin{example}[Young tableaux]
  Like before, we will think of an object of $(\Surj^<)^{nc+}$ as a diagram like the one in \eqref{eq:tablueax}.
  In this case, the fibers are ordered, so the diagram is a Young tableau.
  \begin{equation}
    \label{eq:tablueax}
    \begin{tikzpicture}[baseline={(current bounding box.center)}]
      \node [style=none] (0) at (1.75, 0.5) {$1$};
      \node [style=none] (1) at (-0.75, 0.5) {$6$};
      \node [style=none] (2) at (1.25, 0.5) {$3$};
      \node [style=none] (3) at (-1.75, 0.5) {$4$};
      \node [style=none] (4) at (0.75, 0.5) {$5$};
      \node [style=none] (5) at (-0.25, 0.5) {$7$};
      \node [style=none] (6) at (0.25, 0.5) {$2$};
      \node [style=none] (7) at (-1.25, 0.5) {$8$};
      \node [style=none] (8) at (-1.5, -0.5) {$1$};
      \node [style=none] (9) at (-0.25, -0.5) {$2$};
      \node [style=none] (10) at (1, -0.5) {$3$};
      \node [style=none] (11) at (1.75, -0.5) {$4$};
      \node [style=none] (12) at (-1.75, 1) {};
      \node [style=none] (13) at (-1.25, 1) {};
      \node [style=none] (14) at (-0.75, 1) {};
      \node [style=none] (15) at (0.25, 1) {};
      \node [style=none] (16) at (0.75, 1) {};
      \node [style=none] (17) at (1.25, 1) {};
      \draw (3) to (8);
      \draw (7) to (8);
      \draw (5) to (9);
      \draw (6) to (9);
      \draw (1) to (9);
      \draw (2) to (10);
      \draw (0) to (11);
      \draw (4) to (10);
      \draw [style=Arrow] (12.center) to (13.center);
      \draw [style=Arrow] (14.center) to (15.center);
      \draw [style=Arrow] (16.center) to (17.center);
    \end{tikzpicture}
    \hspace{2cm}
    \ytableausetup{notabloids}
    \begin{ytableau}
      4 & 8 \\
      6 & 7 & 2 \\
      5 & 3 \\
      1
    \end{ytableau}
  \end{equation}
\end{example}

\subsubsection{Descending row condition}

Typically, Young tableaux are written so that the length of the rows decreases from top to bottom.
Consequently, the following tableaux are considered distinct:
\begin{equation}
  \ytableausetup{notabloids}
  \begin{ytableau}
    1 & 2 \\
    3 & 4 \\
    5
  \end{ytableau}
  \hspace{2cm}
  \ytableausetup{notabloids}
  \begin{ytableau}
    3 & 4 \\
    1 & 2 \\
    5
  \end{ytableau}
\end{equation}
We will formulate this condition in terms of the underlying groupoid of $(\Surj^{<})^{nc+}$ and $(\Surj)^{nc+}$.

\begin{definition}
  Given a function $f: X \to Y$ in the category of finite sets, we will say the automorphism $\sigma: Y \to Y$ is \emph{rigid} if $|f^{-1}(y)| = |(\sigma \circ f)^{-1}(y)|$ implies that $\sigma(y) = y$.
  \begin{enumerate}
  \item Let $\mathcal{YT}ableaux$ be the subcategory of $(\Surj^{<})^{nc+}$ whose morphisms are isomorphism $(\sigma \Downarrow \sigma')$ such that $\sigma'$ is rigid.
  \item Let $\mathcal{YT}abloids$ be the subcategory of $(\Surj)^{nc+}$ whose morphisms are isomorphism $(\sigma \Downarrow \sigma')$ such that $\sigma'$ is rigid.
  \end{enumerate}
\end{definition}

\begin{example}
  The descending row condition can be understood as restricting to these subgroupoids.
  For example, the left morphism of \eqref{eq:rigid-and-not} which exchanges the first and second row is in $\mathcal{YT}ableaux$.
  However, the right morphism of \eqref{eq:rigid-and-not} which exchanges the first and third row is not in $\mathcal{YT}ableaux$.
  \begin{equation}
    \label{eq:rigid-and-not}
    \ytableausetup{notabloids}
    \begin{ytableau}
      6 & 7 & 2 \\
      4 & 8 \\
      5 & 3 \\
      1
    \end{ytableau}
    \longleftarrow
    \begin{ytableau}
      4 & 8 \\
      6 & 7 & 2 \\
      5 & 3 \\
      1
    \end{ytableau}
    \longrightarrow
    \begin{ytableau}
      5 & 3 \\
      6 & 7 & 2 \\
      4 & 8 \\
      1
    \end{ytableau}
  \end{equation}

\end{example}

\begin{prop}
  The automorphism group of a particular Young tableau in $\mathcal{YT}ableaux$ is trivial.
\end{prop}
\begin{proof}
  Let $(\sigma \Downarrow \sigma'): f \to f$ be an automorphism in $\mathcal{YT}ableaux$.
  This implies that the following diagram commutes:
  \begin{equation}
    \begin{tikzcd}
      X \arrow[r, "f"] \arrow[d, "\sigma"]
      & Y \arrow[d, "\sigma'"] \\
      X \arrow[r, "f"]
      & Y
    \end{tikzcd}
  \end{equation}
  For each $y \in Y$, the isomorphism $\sigma$ maps the fiber $(\sigma' \circ f)^{-1}(y)$ bijectively onto the fiber $f^{-1}(y)$.
  Hence $\sigma' = \id$ since $\sigma'$ is assumed to be rigid.
  Moreover, $\sigma$ must be an identity to preserve the order on the fibers.
\end{proof}

\begin{cor}
  Each tableau in $\mathcal{YT}ableaux$ is uniquely isomorphic to a tableau where the width of the rows decreases from top to bottom.
\end{cor}
\begin{proof}
  By repeatedly transposing a wider row below with a narrower row above, we obtain a morphism from a particular tableau to a tableau with rows of decreasing length.
  By the previous proposition, this is necessarily unique.
\end{proof}

\section{Cospans}

We will show how cospans are connected to graphs, properads, and Frobenius algebras.

\begin{definition}
  Define the category $\Cospan$ so that:
  \begin{enumerate}
  \item The objects are finite sets.
  \item The morphisms are diagrams $S \rightarrow V \leftarrow T$ called \emph{cospans} modulo isomorphisms in the middle:
    
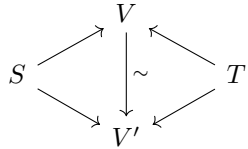
\begin{equation*}
      \begin{tikzcd}[row sep = small]
        & V \arrow[dd, "\sim"] & \\
        S \arrow[ru] \arrow[rd] & & T \arrow[lu] \arrow[ld] \\
        & V' &                        
      \end{tikzcd}
    \end{equation*}
  \item Define composition by taking a pushout in the middle of a pair of cospans:
    
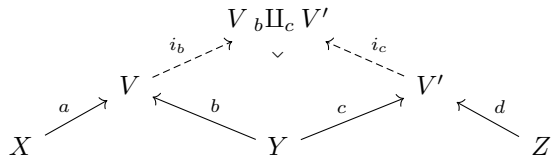
\begin{equation*}
      \begin{tikzcd}[row sep = small]
        && {V \mathbin{_b\amalg_c} V'} \\
        &  V && {V'} \\
        X  && Y && Z
        \arrow["a", from=3-1, to=2-2]
        \arrow["b"', from=3-3, to=2-2]
        \arrow["{i_b}", dashed, from=2-2, to=1-3]
        \arrow["{i_c}"', dashed, from=2-4, to=1-3]
        \arrow["\lrcorner"{anchor=center, pos=0.125, rotate=-45}, draw=none, from=1-3, to=3-3]
        \arrow["c", from=3-3, to=2-4]
        \arrow["d"', from=3-5, to=2-4]
      \end{tikzcd}
    \end{equation*}
  \end{enumerate}
\end{definition}

\subsection{Zero-to-zero morphisms}

One subtlety of cospans is the existence of non-trivial ``zero-to-zero morphisms'' such as $\varnothing \rightarrow pt \leftarrow \varnothing$.
We will discuss how these morphisms differ qualitatively from the others and survey different ways of working with them.

\subsubsection{Subcategories}

In some contexts, the nontrivial zero-to-zero morphisms may not be appropriate.
In these cases, we can safely disregard them by considering different subcategories.

\begin{definition}
  \label{non-unital-cospans}
  For a cospan $S \rightarrow V \leftarrow T$, we refer to $S \rightarrow V$ as the left leg and $V \leftarrow T$ the right leg.
  \begin{enumerate}
  \item Define the subcategory ${}'\Cospan$ of \emph{non--unital cospans} to be the subcategory where the left legs $S \rightarrow V$ are surjective.
  \item Define the subcategory $\Cospan'$ of \emph{non--co--unital cospans} to be the subcategory where the right legs $V \leftarrow T$ are surjective.
  \item Define ${}'\Cospan'$ to be the subcategory where both legs $S \rightarrow V$ and $V \leftarrow T$ are surjective.
  \end{enumerate}
  Because pushouts preserve epimorphisms, these all form categories.
\end{definition}

\begin{example}
  Cospans of the form $S \rightarrow \{\ast\} \leftarrow T$ fail to form a subcategory of $\Cospan$.
  For example, the composition of $S \rightarrow \{\ast\} \leftarrow \varnothing$ and $\varnothing \rightarrow \{\ast\} \leftarrow T$ yields $S \rightarrow \{\ast\} \amalg \{\ast\} \leftarrow T$.
  However, in the cospan categories of Definition~\ref{non-unital-cospans}, the composition of $S \rightarrow \{\ast\} \leftarrow T$ and $T \rightarrow \{\ast\} \leftarrow U$ always yields $S \rightarrow \{\ast\} \leftarrow U$ because of the surjectivity requirements on one or both legs.
\end{example}

\subsubsection{Factorization and lone morphisms}

When the zero-to-zero morphisms are present, the notion of factorization needs a slight modification.
For example, the cospan $\varnothing \rightarrow V \leftarrow \varnothing$ factors as $\coprod_{v \in V}(\varnothing \rightarrow \{v\} \leftarrow \varnothing)$ with $\varnothing \rightarrow \{\ast\} \leftarrow \varnothing$ as the generating morphism.
However, if we consider this cospan as an object of $\Iso(\Cospan \downarrow \Cospan)$, the automorphism group of $\varnothing \rightarrow V \leftarrow \varnothing$ is trivial rather than $\mathrm{Aut}(V)$.
This means the first condition of a UFC does not quite capture the factorization, but it can with a simple modification.

\begin{nota}
  Consider $(\N, +)$ as a discrete monoidal category and let $\N[I]$ denote the monoidal subcategory of $\N^I$ where the objects have finite support.
\end{nota}

\begin{definition}
  Augment the data of a unique factorization category with a set $L$ and an injection $L \rightarrowtail \Hom(1_{\M}, 1_{\M})$.
  Using the monoidal structure of $\M$ and the inclusion $\Iso(1_{\M} \downarrow 1_{\M}) \hookrightarrow \Iso(\M \downarrow \M)$, these induce the following functor.
  \begin{equation}
    \jmath^{\boxtimes} \times l: \Indec^{\boxtimes} \times \N[L] \to \Iso(\M \downarrow \M)
  \end{equation}
  We can think of the image of $L \rightarrowtail \Hom(1_{\M}, 1_{\M})$ as the irreducible \emph{lone morphisms}.
\end{definition}

\begin{example}
  \label{ex:cospan-Eckmann-Hilton}
  For $\Cospan$, take $L = \{ \ast \}$ and define $L \rightarrowtail \Hom(\varnothing, \varnothing)$ so that $\ast$ is assigned to $\varnothing \rightarrow \{\ast\} \leftarrow \varnothing$.
  By the Eckmann--Hilton argument, the only way these lone morphisms can compose is by accumulating.
  For example in $\Cospan$, the composition of $\varnothing \rightarrow \{1\} \leftarrow \varnothing$ and $\varnothing \rightarrow \{2\} \leftarrow \varnothing$ is the cospan $\varnothing \rightarrow \{1\} \amalg \{2\} \leftarrow \varnothing$.
\end{example}

\subsubsection{Corelations}

In applied situations, a morphism in a cospan category is thought of as a process where the source represents the inputs and the target represents outputs.
Under this interpretation, a process with zero inputs and zero outputs should be disregarded whenever they occur.
This is formalized by corelations.

\begin{definition}[Fong--Zanasi~\cite{co-relation}]
  Define the \emph{(first) category of corelations} $\Corel_{I}$ as follows:
  \begin{enumerate}
  \item The objects are sets.
  \item The morphisms are isomorphism classes of jointly surjective cospans.
    In other words, cospans $S \rightarrow V \leftarrow T$ such that the induced map $S \amalg T \to V$ is a surjection.
  \item In general, a composition of two jointly surjective cospans might compose to some $S \amalg T \to V$ that is not jointly surjective.
    However, we can restrict the codomain to get a surjection $S \amalg T \ta V'$.
    We take this to be the composition in $\Corel_I$.
  \end{enumerate}
\end{definition}

\begin{definition}[Fong--Zanasi~\cite{co-relation}]
  Let $\mathcal M$ denote the collection of injections.
  Define the \emph{(second) category of corelations} $\Corel_{II}$ so that the morphisms are equivalence classes of cospans where two cospans are considered equivalent if there is a zig-zag of morphisms in $\mathcal M$ connecting them:
  \begin{equation}
    \begin{tikzcd}
      S \arrow[r, equals] \arrow[d]
      & S \arrow[r, equals] \arrow[d]
      & \ldots \arrow[r, equals]
      & S \arrow[r, equals] \arrow[d]
      & S \arrow[d] \\
      V_1 \arrow[r, "\in \mathcal M"]
      & V_2
      & \ldots \arrow[l, "\in \mathcal M"'] \arrow[r, "\in \mathcal M"]
      & V_{n-1}
      & V_n \arrow[l, "\in \mathcal M"'] \\
      T \arrow[r, equals] \arrow[u]
      & T \arrow[r, equals] \arrow[u]
      & \ldots \arrow[r, equals]
      & T \arrow[r, equals] \arrow[u]
      & T \arrow[u]
    \end{tikzcd}
  \end{equation}
\end{definition}

\begin{example}
  The span $S \overset{l}{\rightarrow} V \overset{r}{\leftarrow} T$ is equivalent to a span $S \rightarrow im(l) \cup im(r) \leftarrow T$ without any ``lone vertices'' by way of the injection $im(l) \cup im(r) \hookrightarrow V$.
\end{example}

\begin{prop}[Fong--Zanasi~\cite{co-relation}]
  \label{Corels-are-equivalent}
  The categories $\Corel_I$ and $\Corel_{II}$ defined above are equivalent.
\end{prop}

\begin{rmk}
  Fong and Zanasi~\cite{co-relation} prove a more general result which shows that the definitions of $\Corel_I$ and $\Corel_{II}$ are valid in any category $\mathcal C$ with pushouts and a factorization system $(\mathcal E, \mathcal M)$ such that $\mathcal M$ is stable under pushouts.
  Moreover, they show that the equivalence of Proposition~\ref{Corels-are-equivalent} still holds.
\end{rmk}

\subsection{Properads and Frobenius algebras}

\subsubsection{Cospan corepresents special Frobenius algebras}
\label{sec:cospan-corep-frobenius}

A strong monoidal functor $A: \Cospan \to \mathcal{C}$ determines an object $A = A(pt)$, a multiplication $\mu = A(\{1, 2\} \rightarrow \{1\} \leftarrow \{1\})$, and a comultiplication $\Delta = A(\{1\} \rightarrow \{1\} \leftarrow \{1,2\})$.
Since the legs of a cospan are set maps, $\mu$ and $\Delta$ must be commutative.
The Frobenius $N$-condition is a consequence of composing cospans by pushouts:
\begin{equation}
  \begin{tikzcd}[sep = small]
    \bullet \arrow[d] & & \bullet \arrow[d] & & \bullet \arrow[rdd] & & \bullet \arrow[ldd] \\
    \bullet & & \bullet & & & & \\
    \bullet \arrow[u] \arrow[d] & \bullet \arrow[lu] \arrow[rd] & \bullet \arrow[d] \arrow[u] & = & & \bullet & \\
    \bullet & & \bullet & & & & \\
    \bullet \arrow[u] & & \bullet \arrow[u] & & \bullet \arrow[ruu] & & \bullet \arrow[luu]
  \end{tikzcd}
\end{equation}

However, composition by pushout automatically implies that $A \overset{\Delta}{\to} A \otimes A \overset{\mu}{\to} A$ is the identity, see Diagram~\eqref{eq:special-law}.
Hence $\Cospan$-ops have an additional property that is not guaranteed by the usual axioms for a Frobenius algebra.
Instead, $\Cospan$ corepresents what are called \emph{special Frobenius algebras} which were first identified by Carboni and Walters in \cite{Carboni-Walters}.
If one uses corelations, one gets the \emph{extra special Frobenius algebras} as described by Coya and Fong in \cite{Extraspecial}.
\begin{equation}
  \label{eq:special-law}
  \begin{tikzcd}[sep = small]
    & \bullet \arrow[d] & & & \bullet \arrow[dd] \\
    & \bullet & & & \\
    \bullet \arrow[rd] \arrow[ru] & & \bullet \arrow[ld] \arrow[lu] & = & \bullet \\
    & \bullet & & & \\
    & \bullet \arrow[u] & & & \bullet \arrow[uu]
  \end{tikzcd}
\end{equation}

\subsubsection{Genus data}
\label{sec:genus-data}

Considering the equivalence of commutative Frobenius algebras and 2D topological quantum field theories proved by Abrams~\cite{Abrams}, we see that $\Cospan$ would need to be equipped with an extra ``genus datum'' in order to corepresent commutative Frobenius algebras.
To understand the nature of this genus datum, we will consider the following connection between cospans and the properads of Vallette~\cite{Vallette}.

\begin{prop}
  \label{cospan-plus-as-graph}
  \cite{KMoManin}
  Combinatorially, $\Cospan^+$ is equivalent to a Borisov--Manin category of graphs such that:
  \begin{enumerate}
  \item The objects are directed aggregates of corollas.
  \item The morphisms are generated by monoidal products and compositions of morphisms with connected two-level ghost graphs.
  \end{enumerate}
  Moreover, $\Cospan^+$ corepresents properads as a Feynman category.
\end{prop}
\begin{proof}[Sketch]
  The objects of $\Cospan^+$ are cospans $S \rightarrow V \leftarrow T$.
  Let $F = S \amalg T$ be the flags with the elements of $S$ as the in-flags and the elements of $T$ as the out-flags.
  Let $V$ be the vertices.
  The universal property of coproducts gives a map $\del: S \amalg T \to V$.
  This determines a Borisov--Manin graph with no edges.

  A typical generating morphism in $\Cospan^+$ is depicted in Diagram~\eqref{eq:typical-morphism-cospan-plus}.
  Under the graphical interpretation of cospans, the $\gamma$-morphisms of the plus construction separate the vertices into two levels and match the out-flags of the top vertices to the in-flags of the bottom vertices.
  \begin{equation}
    \label{eq:typical-morphism-cospan-plus}
    \begin{tikzpicture}[baseline=(current  bounding  box.center)]
      \node [style=White] (0) at (-3, 0) {};
      \node [style=White] (1) at (-2, 0) {};
      \node [style=White] (2) at (-1, 0) {};
      \node [style=none] (3) at (-3.25, -0.5) {};
      \node [style=none] (4) at (-3, 0.5) {};
      \node [style=none] (5) at (-2, 0.5) {};
      \node [style=none] (6) at (-1.25, -0.5) {};
      \node [style=none] (7) at (-1.25, 0.5) {};
      \node [style=none] (8) at (-0.75, 0.5) {};
      \node [style=White] (18) at (2.5, 0) {};
      \node [style=none] (19) at (2, -0.5) {};
      \node [style=none] (20) at (2, 0.5) {};
      \node [style=none] (21) at (2.5, 0.5) {};
      \node [style=none] (22) at (3, 0.5) {};
      \node [style=none] (23) at (-0.25, 0) {};
      \node [style=none] (24) at (1.5, 0) {};
      \node [style=none] (25) at (-2.75, -0.5) {};
      \node [style=none] (26) at (-2, -0.5) {};
      \node [style=none] (27) at (-0.75, -0.5) {};
      \node [style=none] (28) at (2.5, -0.5) {};
      \node [style=White] (30) at (-4, 0) {};
      \node [style=none] (31) at (-4.25, 0.5) {};
      \node [style=none] (32) at (-3.75, 0.5) {};
      \node [style=none] (33) at (-4, -0.5) {};
      \node [style=SmallWhite] (50) at (0.75, 0.75) {};
      \node [style=SmallWhite] (51) at (0.25, 1) {};
      \node [style=SmallWhite] (52) at (0.75, 1) {};
      \node [style=none] (53) at (0.5, 0.5) {};
      \node [style=none] (55) at (0.25, 1.25) {};
      \node [style=none] (57) at (0.5, 1.25) {};
      \node [style=none] (58) at (1, 1.25) {};
      \node [style=none] (59) at (1, 0.5) {};
      \node [style=SmallWhite] (62) at (0.25, 0.75) {};
      \node [style=none] (65) at (0.25, 0.5) {};
      \node [style=none] (66) at (3, -0.5) {};
      \draw (3.center) to (0);
      \draw (4.center) to (0);
      \draw (5.center) to (1);
      \draw (6.center) to (2);
      \draw (7.center) to (2);
      \draw (8.center) to (2);
      \draw (19.center) to (18);
      \draw (20.center) to (18);
      \draw (21.center) to (18);
      \draw (22.center) to (18);
      \draw [style=Arrow] (23.center) to (24.center);
      \draw (0) to (25.center);
      \draw (1) to (26.center);
      \draw (2) to (27.center);
      \draw (18) to (28.center);
      \draw (31.center) to (30);
      \draw (32.center) to (30);
      \draw (30) to (33.center);
      \draw (53.center) to (50);
      \draw (55.center) to (51);
      \draw (57.center) to (52);
      \draw (58.center) to (52);
      \draw (50) to (59.center);
      \draw (62) to (65.center);
      \draw (62) to (52);
      \draw (52) to (50);
      \draw (51) to (62);
      \draw (66.center) to (18);
    \end{tikzpicture}
  \end{equation}
  The assumption that the generating morphisms are connected excludes graph mergers.
\end{proof}

\begin{definition}[Genus properad]
  This graphical description is convenient and nicely complements the work of Berger and Kaufmann in \cite{DDecDennis} where they give a categorical and combinatorial account of different structures and operations coming from string topology.

  In particular, they demonstrate that the genus datum can be understood as a strong monoidal functor $\O_{\mathrm{genus}}: (\Graphs, \amalg) \to (\Set, \times)$ which assigns a genus labeling to each vertex.
  Specializing to our situation, define the \emph{genus properad} $\O_{\mathrm{genus}}^*: \Cospan^+ \to (\Set, \times)$ to be the strong monoidal functor obtained by the following composition:
  \begin{equation}
    \Cospan^+ \hookrightarrow (\Graphs^{\mathrm{dir}}, \amalg) \ta (\Graphs, \amalg) \overset{\O_{\mathrm{genus}}}{\to} (\Set, \times)
  \end{equation}
\end{definition}

\subsubsection{Indexed enrichments}

With the extended notion of a plus construction described in \cite{KMoManin}, the indexed enrichments of Kaufmann and Ward \cite{feynman} can be adapted to unique factorization categories.
We will describe this briefly here and use it to establish the desired connection to Frobenius algebras.

\begin{definition}[Kaufmann--Ward~\cite{feynman}]
  Given a strong monoidal functor $D: \M^+ \to \mathcal{C}$, define the (enriched) category $\M_{\O}$ as follows:
  \begin{enumerate}
  \item The objects are the same as the original $\M$.
  \item The hom $\mathcal{C}$-objects are $\Hom_{\M_{\O}}(X, Y) = \coprod_{\phi \in \Hom_{\M}(X, Y)} D(\phi)$.
  \item The composition is induced by the $\gamma$-morphisms $D(\phi) \otimes D(\psi) \to D(\phi \circ \psi)$ of the plus construction.
  \end{enumerate}
\end{definition}

\begin{example}
  The indexed enrichment of an operad $\O: \FinSet^+ \to \mathcal{C}$ produces a new category $\FinSet_{\O}$ which corepresents $\O$-algebras.
  For more examples of this sort, we refer the reader to \cite{feynmanrep}.
\end{example}

\begin{cor}
  The category $\Cospan_{\O_{\mathrm{genus}}^*}$ corepresents commutative Frobenius monoids.
\end{cor}
\begin{proof}
  With the extra genus datum, the morphisms of $\Cospan_{\O_{\mathrm{genus}}^*}$ agree with the Abrams description of Frobenius algebras as 2D topological quantum field theories.
\end{proof}

\subsection{Structured cospans}

There are several methods for constructing cospan categories equipped with extra structure such as the decorated cospans of Fong~\cite{decorated-cospans} and the structured cospans of Baez and Courser~\cite{baez2019structured, courser2020open}.
In this section, we will briefly survey structured cospans and use it to describe colored versions of cospans and properads.

\begin{definition}[Courser~\cite{courser2020open}]
  Given a functor $F: \Gpd \to \mathcal{C}$ where $\mathcal{C}$ has pushouts, there is a double category $\Cospan(F)$ of \emph{structured cospans}:
  \begin{enumerate}
  \item The objects are the same as $\Gpd$ and the vertical morphisms are the morphisms of $\Gpd$.
  \item The horizontal morphisms are cospans with $F$ applied to both feet, see Diagram~\eqref{eq:hor-arrows}.
    Horizontal composition is given by taking pushouts.
    \begin{equation}
      \label{eq:hor-arrows}
      \begin{tikzcd}
        X & Y && {F(X)} & V & {F(Y)} \\
        {X'} & {Y'} && {F(X')} & {V'} & {F(Y')}
    	\arrow["\sigma"', from=1-1, to=2-1]
    	\arrow[""{name=0, anchor=center, inner sep=0}, "\tau", from=1-2, to=2-2]
    	\arrow[""{name=1, anchor=center, inner sep=0}, "\shortmid"{marking}, from=1-1, to=1-2]
    	\arrow[""{name=2, anchor=center, inner sep=0}, "\shortmid"{marking}, from=2-1, to=2-2]
        \arrow[""{name=3, anchor=center, inner sep=0}, "{\sigma}"', from=1-4, to=2-4]
    	\arrow["{\tau}", from=1-6, to=2-6]
        \arrow[from=2-4, to=2-5]
        \arrow[from=2-6, to=2-5]
        \arrow[from=1-4, to=1-5]
        \arrow[from=1-6, to=1-5]
        \arrow["\Phi", from=1-5, to=2-5]
        \arrow["\Phi", shorten <=10pt, shorten >=10pt, Rightarrow, from=1, to=2]
    	\arrow["{:=}"{description}, Rightarrow, draw=none, from=0, to=3]
      \end{tikzcd}
    \end{equation}
  \end{enumerate}
\end{definition}

\begin{remark}
  The functor $F$ is denoted by $L$ in Baez and Courser~\cite{baez2019structured} which stands for ``left adjoint'' since the functor often is indeed a left adjoint in their work.
  For us, this is generally not the case, so we drop that convention.
\end{remark}

\begin{prop}[Section 3.4 of \cite{courser2020open}]
  If $\Gpd$ and $\mathcal{C}$ are symmetric monoidal categories and $F: \Gpd \to \mathcal{C}$ is a strong symmetric monoidal functor, then the double category $\Cospan(F)$ becomes symmetric monoidal in a canonical way.
\end{prop}

\begin{figure}
    \centering
    \begin{tikzpicture}
      \node [style=Red] (1) at (-4, 3) {};
      \node [style=Blue] (2) at (-4, 2) {};
      \node [style=Blue] (3) at (-4, 1) {};
      \node [style=Red] (5) at (1, 3) {};
      \node [style=Red] (6) at (1, 1) {};
      \node [style=Red] (8) at (-1.5, 3) {};
      \node [style=Blue] (9) at (-1.5, 2) {};
      \node [style=Red] (10) at (-1.5, 1) {};
      \node [style=WhiteSq,minimum size=0.4cm] (13) at (-2.5, 2.5) {};
      \node [style=WhiteSq,minimum size=0.4cm] (14) at (-2.5, 1.5) {};
      \node [style=WhiteSq,minimum size=0.4cm] (16) at (-0.5, 2) {};
      \node [style=none] (18) at (-3, 3.5) {};
      \node [style=none] (19) at (0, 3.5) {};
      \node [style=none] (20) at (0, 0.5) {};
      \node [style=none] (21) at (-3, 0.5) {};
      \node [style=none] (38) at (2, 5) {};
      \node [style=none] (39) at (2, 0) {};
      \node [style=none] (40) at (-1.5, 4.75) {\large Composition};
      \node [style=none] (41) at (4.5, 4.75) {\large Result};
      \node [style=Red] (42) at (3, 3) {};
      \node [style=Blue] (43) at (3, 2) {};
      \node [style=Blue] (44) at (3, 1) {};
      \node [style=Red] (45) at (6, 3) {};
      \node [style=Red] (46) at (6, 1) {};
      \node [style=WhiteSq,minimum size=0.4cm] (47) at (4.5, 2) {};
      \draw (13) to (8);
      \draw (13) to (9);
      \draw (9) to (16);
      \draw (16) to (5);
      \draw (6) to (16);
      \draw (1) to (13);
      \draw (2) to (14);
      \draw (14) to (10);
      \draw (3) to (14);
      \draw [style=Dash] (18.center) to (21.center);
      \draw [style=Dash] (20.center) to (21.center);
      \draw [style=Dash] (20.center) to (19.center);
      \draw [style=Dash] (19.center) to (18.center);
      \draw (38.center) to (39.center);
      \draw (8) to (16);
      \draw (10) to (16);
      \draw (42) to (47);
      \draw (47) to (45);
      \draw (43) to (47);
      \draw (44) to (47);
      \draw (46) to (47);
    \end{tikzpicture}
    \caption{A composition of two morphisms in the colored cospan category.}
        \label{fig-colored-cospan}
\end{figure}
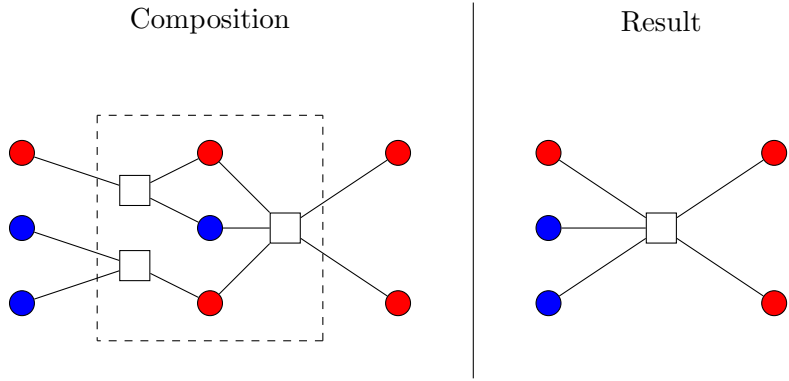

\begin{ex}[Cospans and properads with colors]
    \label{colored-cospans}
    Let $\V$ be a discrete category with two objects $R$ (``red'') and $B$ (``blue'').
    There is a canonical strong monoidal functor $clr: \V^{\otimes} \to \FinSet$ which sends an object of length $n$ to the set $\underline{n}$.
    By the previous result, this defines a symmetric monoidal category $\Cospan(clr)$.
    Composition is illustrated in Figure~\ref{fig-colored-cospan}.
    Moreover, $\Cospan(clr)^+$ corepresents 2-colored properads.
\end{ex}

\subsection{Props and mergers}

We always have a localization functor $L_{\M}: \M^{nc+} \to \M^{+}$ for any monoidal category $\M$ which sends a basic object $f_1 \boxtimes \ldots \boxtimes f_n$ to $f_1 \otimes \ldots \otimes f_n$.
If there is another functor $J: \M^{nc+} \to \M^{+}$, we can incorporate these as an extra structure on a plus construction which allows ``mergers'' of morphisms.

\begin{definition}
  Given a functor $J: \M^{nc+} \to \M^{+}$, define a new category by starting with $\M^+$ then formally add a generating morphism $B_\phi: L_{\M}(\phi) \to J(\phi)$ for each basic object $\phi$ in $\M^{nc+}$.
  We then add relations making $B_\phi$ natural in the sense that the following diagram commutes for each morphism $\Phi$ of $\M^{nc+}$:
  \begin{equation}
    \begin{tikzcd}
      L_{\M}(\phi)
      \arrow[d, "B_\phi"']
      \arrow[r, "L_{\M}(\Phi)"]
      & L_{\M}(\psi)
      \arrow[d, "B_\psi"] \\
      J(\phi)
      \arrow[r, "J(\Phi)"']
      & J(\psi)
    \end{tikzcd}
  \end{equation}
\end{definition}

\begin{example}
  In the category of cospans, there is a functor $J: Cospan^{nc+} \to Cospan^{+}$ which sends a basic object $S \rightarrow V \leftarrow T$ to the basic object $S \rightarrow pt \leftarrow T$.
  The morphisms $B_{\phi}$ are mergers in the ordinary sense.
  Naturality for isomorphisms is just equivariance.
  Naturality for $\gamma: \phi_0 \boxtimes \phi_1 \to \phi_0 \circ \phi_1$ is a sort of interchange, see Figure~\ref{fig:B-interchange}.  
\end{example}

\begin{figure}
  \centering
  \begin{tikzpicture}
    \node [style=White] (0) at (-5.25, 2) {};
    \node [style=White] (1) at (-4.5, 2) {};
    \node [style=White] (2) at (-3.75, 2) {};
    \node [style=White] (3) at (-2.5, 2) {};
    \node [style=White] (4) at (-1.25, 2) {};
    \node [style=White] (5) at (-2, 2) {};
    \node [style=none] (6) at (-3, 2.75) {};
    \node [style=none] (7) at (-3, 1.25) {};
    \node [style=none] (8) at (-5.25, 1.5) {};
    \node [style=none] (9) at (-4.5, 1.5) {};
    \node [style=none] (10) at (-4, 1.5) {};
    \node [style=none] (11) at (-3.5, 1.5) {};
    \node [style=none] (12) at (-2.5, 1.5) {};
    \node [style=none] (13) at (-1.5, 1.5) {};
    \node [style=none] (14) at (-1, 1.5) {};
    \node [style=none] (15) at (-2, 1.5) {};
    \node [style=none] (16) at (-5.25, 2.5) {};
    \node [style=none] (17) at (-4.75, 2.5) {};
    \node [style=none] (18) at (-4.25, 2.5) {};
    \node [style=none] (19) at (-3.75, 2.5) {};
    \node [style=none] (20) at (-2.5, 2.5) {};
    \node [style=none] (21) at (-1.5, 2.5) {};
    \node [style=none] (22) at (-1, 2.5) {};
    \node [style=none] (23) at (-2, 2.5) {};
    \node [style=White] (25) at (-4.5, -1.5) {};
    \node [style=White] (28) at (-1.75, -1.5) {};
    \node [style=none] (30) at (-3, -0.75) {};
    \node [style=none] (31) at (-3, -2.25) {};
    \node [style=none] (32) at (-5, -2) {};
    \node [style=none] (33) at (-4.75, -2) {};
    \node [style=none] (34) at (-4.25, -2) {};
    \node [style=none] (35) at (-4, -2) {};
    \node [style=none] (36) at (-2.25, -2) {};
    \node [style=none] (37) at (-2, -2) {};
    \node [style=none] (38) at (-1.5, -2) {};
    \node [style=none] (39) at (-1.25, -2) {};
    \node [style=none] (40) at (-5, -1) {};
    \node [style=none] (41) at (-4.75, -1) {};
    \node [style=none] (42) at (-4.25, -1) {};
    \node [style=none] (43) at (-4, -1) {};
    \node [style=none] (44) at (-2.25, -1) {};
    \node [style=none] (45) at (-2, -1) {};
    \node [style=none] (46) at (-1.5, -1) {};
    \node [style=none] (47) at (-1.25, -1) {};
    \node [style=none] (48) at (-0.25, 2) {};
    \node [style=none] (49) at (1.25, 2) {};
    \node [style=none] (50) at (-3, 0.75) {};
    \node [style=none] (51) at (-3, -0.25) {};
    \node [style=none] (52) at (-0.25, -1.5) {};
    \node [style=none] (53) at (1.25, -1.5) {};
    \node [style=White] (54) at (2, 2) {};
    \node [style=White] (55) at (3, 2) {};
    \node [style=none] (57) at (2, 1.5) {};
    \node [style=none] (58) at (2.5, 1.5) {};
    \node [style=none] (59) at (3, 1.5) {};
    \node [style=none] (60) at (3.5, 1.5) {};
    \node [style=none] (61) at (2, 2.5) {};
    \node [style=none] (62) at (2.5, 2.5) {};
    \node [style=none] (63) at (3, 2.5) {};
    \node [style=none] (64) at (3.5, 2.5) {};
    \node [style=White] (66) at (2.75, -1.5) {};
    \node [style=none] (67) at (2.25, -2) {};
    \node [style=none] (68) at (2.5, -2) {};
    \node [style=none] (69) at (3, -2) {};
    \node [style=none] (70) at (3.25, -2) {};
    \node [style=none] (71) at (2.25, -1) {};
    \node [style=none] (72) at (2.5, -1) {};
    \node [style=none] (73) at (3, -1) {};
    \node [style=none] (74) at (3.25, -1) {};
    \node [style=none] (75) at (2.75, 0.75) {};
    \node [style=none] (76) at (2.75, -0.25) {};
    \node [style=SmallWhite] (77) at (0.5, -2.75) {};
    \node [style=none] (78) at (0, -3) {};
    \node [style=none] (79) at (0.25, -3) {};
    \node [style=none] (80) at (0.75, -3) {};
    \node [style=none] (81) at (1, -3) {};
    \node [style=none] (82) at (0, -2.5) {};
    \node [style=none] (83) at (0.25, -2.5) {};
    \node [style=none] (84) at (0.75, -2.5) {};
    \node [style=none] (85) at (1, -2.5) {};
    \node [style=SmallWhite] (86) at (0.5, -2.25) {};
    \node [style=none] (87) at (0, -2) {};
    \node [style=none] (88) at (0.25, -2) {};
    \node [style=none] (89) at (0.75, -2) {};
    \node [style=none] (90) at (1, -2) {};
    \node [style=none] (91) at (0, 2.5) {};
    \node [style=SmallWhite] (92) at (0.5, 2.75) {};
    \node [style=SmallWhite] (93) at (1, 2.75) {};
    \node [style=SmallWhite] (94) at (0, 2.75) {};
    \node [style=none] (95) at (0.5, 2.5) {};
    \node [style=none] (96) at (0.75, 2.5) {};
    \node [style=none] (97) at (1.25, 2.5) {};
    \node [style=SmallWhite] (102) at (0, 3.25) {};
    \node [style=SmallWhite] (103) at (1, 3.25) {};
    \node [style=SmallWhite] (104) at (0.5, 3.25) {};
    \node [style=none] (109) at (0, 3.5) {};
    \node [style=none] (110) at (0.75, 3.5) {};
    \node [style=none] (111) at (1.25, 3.5) {};
    \node [style=none] (112) at (0.5, 3.5) {};
    \draw [dashed] (6.center) to (7.center);
    \draw (16.center) to (0);
    \draw (0) to (8.center);
    \draw (17.center) to (1);
    \draw (1) to (18.center);
    \draw (1) to (9.center);
    \draw (19.center) to (2);
    \draw (2) to (10.center);
    \draw (2) to (11.center);
    \draw (20.center) to (3);
    \draw (3) to (12.center);
    \draw (4) to (21.center);
    \draw (22.center) to (4);
    \draw (4) to (13.center);
    \draw (4) to (14.center);
    \draw (23.center) to (5);
    \draw (5) to (15.center);
    \draw [dashed] (30.center) to (31.center);
    \draw (41.center) to (25);
    \draw (25) to (42.center);
    \draw (25) to (33.center);
    \draw (28) to (45.center);
    \draw (46.center) to (28);
    \draw (28) to (37.center);
    \draw (28) to (38.center);
    \draw (40.center) to (25);
    \draw (25) to (32.center);
    \draw (25) to (34.center);
    \draw (25) to (43.center);
    \draw (25) to (35.center);
    \draw (44.center) to (28);
    \draw (36.center) to (28);
    \draw (39.center) to (28);
    \draw (47.center) to (28);
    \draw [style=Arrow] (48.center) to (49.center);
    \draw [style=Arrow] (52.center) to (53.center);
    \draw [style=Arrow] (50.center) to (51.center);
    \draw (61.center) to (54);
    \draw (54) to (57.center);
    \draw (55) to (62.center);
    \draw (63.center) to (55);
    \draw (55) to (58.center);
    \draw (55) to (59.center);
    \draw (64.center) to (55);
    \draw (60.center) to (55);
    \draw (66) to (72.center);
    \draw (73.center) to (66);
    \draw (66) to (68.center);
    \draw (66) to (69.center);
    \draw (74.center) to (66);
    \draw (70.center) to (66);
    \draw [style=Arrow] (75.center) to (76.center);
    \draw (83.center) to (77);
    \draw (77) to (84.center);
    \draw (77) to (79.center);
    \draw (82.center) to (77);
    \draw (77) to (78.center);
    \draw (77) to (80.center);
    \draw (77) to (85.center);
    \draw (77) to (81.center);
    \draw (86) to (88.center);
    \draw (89.center) to (86);
    \draw (87.center) to (86);
    \draw (90.center) to (86);
    \draw (86) to (82.center);
    \draw (86) to (83.center);
    \draw (86) to (84.center);
    \draw (86) to (85.center);
    \draw (91.center) to (94);
    \draw (92) to (95.center);
    \draw (93) to (96.center);
    \draw (93) to (97.center);
    \draw (109.center) to (102);
    \draw (103) to (110.center);
    \draw (111.center) to (103);
    \draw (112.center) to (104);
    \draw (104) to (92);
    \draw (112.center) to (104);
    \draw (103) to (93);
    \draw (103) to (92);
    \draw (71.center) to (66);
    \draw (66) to (67.center);
    \draw (102) to (94);
  \end{tikzpicture}
  \caption{The naturality condition of $B$ on the $\gamma$-morphisms. The dashed line indicates a $\boxtimes$ that became an $\otimes$ after applying either the functor $L$ or $J$.}
    \label{fig:B-interchange}
\end{figure}
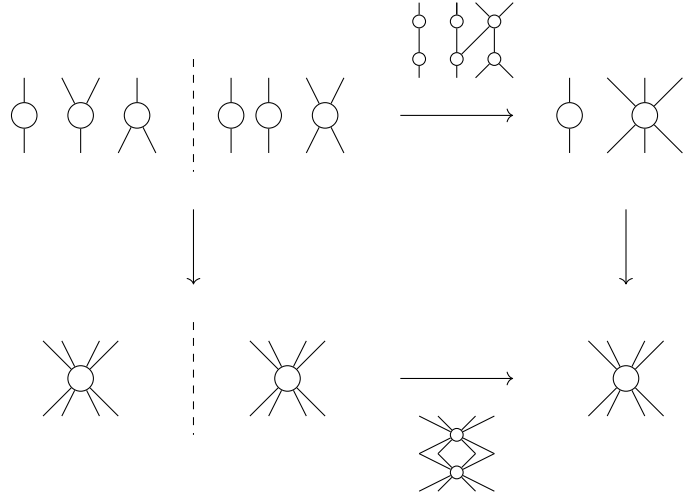

\section{Spans}

Similar to cospans, the category $\Span$ is defined so that:
\begin{enumerate}
\item The objects are finite sets.
\item The morphisms are diagrams $S \leftarrow V \rightarrow T$ called \emph{spans} modulo isomorphisms in the middle:
  
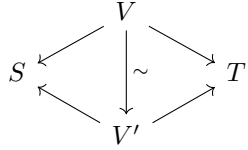
\begin{equation*}
    \begin{tikzcd}[row sep = small]
      & V \arrow[dd, "\sim"] & \\
      S \arrow[ru, leftarrow] \arrow[rd, leftarrow]
      & & T \arrow[lu, leftarrow] \arrow[ld, leftarrow] \\
      & V' &                        
    \end{tikzcd}
  \end{equation*}
\item Define composition by taking a pullback in the middle of a pair of spans:
  \begin{equation*}
    \begin{tikzcd}[row sep = small]
      && {V \mathbin{_b\times_c} V'} \\
      &  V && {V'} \\
      X  && Y && Z
      \arrow["a", from=3-1, to=2-2, leftarrow]
      \arrow["b"', from=3-3, to=2-2, leftarrow]
      \arrow["{\pi_b}", dashed, from=2-2, to=1-3, leftarrow]
      \arrow["{\pi_c}"', dashed, from=2-4, to=1-3, leftarrow]
      \arrow["\lrcorner"{anchor=center, pos=0.125, rotate=-45}, draw=none, from=1-3, to=3-3, leftarrow]
      \arrow["c", from=3-3, to=2-4, leftarrow]
      \arrow["d"', from=3-5, to=2-4, leftarrow]
    \end{tikzcd}
  \end{equation*}
\end{enumerate}

\begin{prop}
  \cite{KMoManin}
  $\Span$ is a hereditary unique factorization category.
\end{prop}

\subsection{Relations}

Classically, a relation between two sets $X$ and $Y$ is a subset $R \subseteq X \times Y$.
From a categorical point of view, an injection $r: R \rightarrowtail X \times Y$ determines a span $X \leftarrow R \rightarrow Y$, so relations can be understood as a subcategory of $\Span$.

\begin{prop}[well-known]
  Let $r_0: R_0 \rightarrowtail X \times Y$ and $r_1: R_1 \rightarrowtail Y \times Z$ be relations with the following associated spans:
  \begin{align*}
    X \overset{a_0}{\leftarrow} R_0 \overset{b_0}{\rightarrow} Y \\
    Y \overset{a_1}{\leftarrow} R_1 \overset{b_1}{\rightarrow} Z 
  \end{align*}
  The composition $(a_1, b_1) \circ (a_0, b_0)$ of spans is the same as the span associated to the composition $r_1 \circ r_0$ of relations.
\end{prop}
\begin{proof}
  Consider the composition of $X \overset{a_0}{\leftarrow} R_0 \overset{b_0}{\rightarrow} Y$ and $Y \overset{a_1}{\leftarrow} R_1 \overset{b_1}{\rightarrow} Z$.
  Let $R_2 = \{ (x_0, x_1) \in R_0 \times R_1 : b_0(x_0) = a_1(x_1) \}$ so that $R_0 \overset{\pi_0}{\leftarrow} R_2 \overset{\pi_1}{\rightarrow} R_1$ forms the pullback of $R_0 \overset{b_0}{\rightarrow} Y \overset{a_1}{\leftarrow} R_1$.
  Hence the following span represents the composition:
  \begin{equation}
    X
    \overset{a_0}{\leftarrow}
    R_0
    \overset{\pi_0}{\leftarrow}
    R_2
    \overset{\pi_1}{\rightarrow}
    R_1
    \overset{b_1}{\rightarrow}
    Z
  \end{equation}
  This gives $r_2: R_2 \rightarrowtail R_0 \times R_1 \rightarrowtail X \times Z$.
  The relation $r_2$ is precisely the composition $r_1 \circ r_0$.
\end{proof}

\begin{prop}
  The category of relations is a hereditary unique factorization category.
\end{prop}
\begin{proof}
  Since relations form a monoidal subcategory of $\Span$, this result follows from the fact that $\Span$ is a unique factorization category.
\end{proof}

\subsection{Matrices}
\label{sec:matrices}

We can also think of $\Span$ as a categorified version of a matrix.
To see this, note that each span $S \leftarrow V \rightarrow T$ corresponds to a map $M: V \to S \times T$ by the universal property of the product.
The map $M$ is determined by its fibers $M_{s,t} = M^{-1}(s,t)$, so $M$ can be understood as a sort of matrix where the coordinates are sets rather than numbers.
This analogy is strengthened by the following fact.

\begin{prop}[well-known]
  Given the maps $M: V \to A \times B$ and $N: U \to B \times C$, let $M \circ N$ denote their composition as a pair of spans.
  Then the coordinates of $M \circ N$ are given by a categorified matrix multiplication:
  \begin{equation}
    (M \circ N)_{a,c}
    \cong \coprod_{b \in B} M_{a,b} \times N_{b,c}
  \end{equation}
\end{prop}
\begin{proof}
  Let $A \leftarrow V \rightarrow B$ be the span for $M$ and let $B \leftarrow U \rightarrow C$ be the span for $N$.
  Write the pullback as $V \leftarrow W \rightarrow U$.
  As shorthand, we will use subscripts for preimages.
  For example, $V_a$ is the preimage of $a \in A$ under the map $V \to A$.
  We will also use double subscripts for intersections, so $V_{ab} = V_a \cap V_b$.

  First, a straightforward pullback argument shows that the 2-by-2 square below is a pullback:
  \begin{equation}
    \begin{tikzcd}[sep = small]
      W_{ac} \arrow[r] \arrow[d] & W_a \arrow[r] \arrow[d] & V_a \arrow[r] \arrow[d] & \{a\} \arrow[d] \\
      W_c \arrow[r] \arrow[d]  & W \arrow[d] \arrow[r] & V \arrow[r] \arrow[d] & A \\
      U_c \arrow[d] \arrow[r]    & U \arrow[d] \arrow[r]   & B                       &                 \\
      \{c\} \arrow[r]            & C                       &                         &                
    \end{tikzcd}
  \end{equation}
  Then for each point $b \in B$, the dashed arrow is uniquely determined by the universal property of the pullback:
  \begin{equation}
    \label{cube}
    \begin{tikzcd}[sep=tiny]
      V_{ab} \times U_{bc} \arrow[dd] \arrow[rr] \arrow[rd, dashed]
      & & V_{ab} \arrow[rd] \arrow[dd] & \\
      & W_{ac}
      \arrow[rr, crossing over]
      & & V_a \arrow[dd] \\
      U_{bc} \arrow[rr] \arrow[rd]
      & & \{b\} \arrow[rd] & \\
      & U_c \arrow[rr] \arrow[from=uu, crossing over] & & B
    \end{tikzcd}
  \end{equation}

  Now consider the diagrams below.
  Since (1) and (2) are pullbacks, the rectangle (1,2) is also a pullback.
  By commutativity of \eqref{cube}, rectangle (1,2) is the same as (3,4).
  Since (3,4) and (4) are pullbacks, (3) is also a pullback.
  \begin{equation}
    \begin{tikzcd}[column sep = small]
      {V_{ab} \times U_{bc}} & {V_{ab}} & {V_a} & {V_{ab} \times U_{bc}} & {W_{ac}} & {V_a} \\
      {U_{bc}} & {\{b\}} & B & {U_{bc}} & {U_c} & B
      \arrow[from=1-2, to=2-2]
      \arrow[from=1-2, to=1-3]
      \arrow[from=2-2, to=2-3]
      \arrow[from=1-3, to=2-3]
      \arrow[from=1-1, to=1-2]
      \arrow[from=2-1, to=2-2]
      \arrow[from=1-1, to=2-1]
      \arrow[from=1-4, to=2-4]
      \arrow[from=2-4, to=2-5]
      \arrow[from=2-5, to=2-6]
      \arrow[from=1-6, to=2-6]
      \arrow[from=1-5, to=2-5]
      \arrow[from=1-5, to=1-6]
      \arrow[dashed, from=1-4, to=1-5]
      \arrow["{(1)}"{description}, draw=none, from=1-1, to=2-2]
      \arrow["{(2)}"{description}, draw=none, from=1-2, to=2-3]
      \arrow["{(3)}"{description}, draw=none, from=1-4, to=2-5]
      \arrow["{(4)}"{description}, draw=none, from=1-5, to=2-6]
    \end{tikzcd}
  \end{equation} 
  Diagram (5) below is a pullback, hence the diagram (3,5) is a pullback for all $b \in B$.
  \begin{equation}
    \label{eq:3}
    \begin{tikzcd}
      {V_{ab} \times U_{bc}} & {U_{bc}} & {\{b\}} \\
      {W_{ac}} & {U_c} & B
      \arrow[from=1-2, to=1-3]
      \arrow[from=1-3, to=2-3]
      \arrow[from=1-2, to=2-2]
      \arrow[from=2-2, to=2-3]
      \arrow[from=2-1, to=2-2]
      \arrow[from=1-1, to=1-2]
      \arrow["{(5)}"{description}, draw=none, from=1-2, to=2-3]
      \arrow[from=1-1, to=2-1]
      \arrow["{(3)}"{description}, draw=none, from=1-1, to=2-2]
    \end{tikzcd}    
  \end{equation}
  By pullback stability, the following diagram is also a pullback.
  \begin{equation}
    \begin{tikzcd}
      \coprod_{b \in B}V_{ab} \times U_{bc} \arrow[d] \arrow[r] & \coprod_{b \in B}\{b\} \arrow[d] \\
      W_{ac} \arrow[r] & B                               
    \end{tikzcd}  
  \end{equation}
  Since pullbacks preserve isomorphisms, we get $W_{ac} \cong \coprod_{b \in B}V_{ab} \times U_{bc}$, as desired.
\end{proof}

\subsection{Black and white graphs}
\label{sec:black-white-graphs}

We can also think of the span $V_w \overset{w}{\leftarrow} E \overset{b}{\rightarrow} V_b$ as a black and white graph where the edge $e \in E$ is connected to a white vertex $w(e) \in V_w$ and a black vertex $b(e) \in V_b$.
To define a composition of these black and white graphs, suppose we have the following black and white graphs such that $V_b = V'_w$:
\begin{center}
  \begin{tabular}{ccc}
    $\Gamma = (V_w \leftarrow E \rightarrow V_b)$
    & and
    & $\Gamma' = (V'_w \leftarrow E' \rightarrow V'_b)$ 
  \end{tabular}  
\end{center}
Define $\Gamma \wedge \Gamma'$ to be the following pullback:
\begin{equation}
  \begin{tikzcd}
    & & E^{\wedge} = E {}_{\del_b}{\times}{}_{\del'_w} E' \arrow[ld] \arrow[rd] & & \\
    & E \arrow[ld] \arrow[rd] & & E' \arrow[ld] \arrow[rd] & \\
    V_w & & V_b = V'_w & & V'_b
  \end{tikzcd}
\end{equation}

This new graph has $V_w$ as its white vertices and $V'_b$ as its black vertices.
There is an edge between vertices $w \in V_w$ and $b' \in V'_b$ in the new graph whenever there is a vertex $x \in V_b = V'_w$ with an edge between $x$ and $w$ in the first graph and an edge between $x$ and $b'$ in the second graph.

\begin{ex}
  Consider the $\wedge$-operation depicted in \eqref{orion-composition} between a simplified version of the constellation Orion with edges labeled by letters and another black and white graph with edges labeled by numbers.
  The black and white vertices that are matched are depicted as squares.
  \begin{equation}
    \label{orion-composition}
    \resizebox{0.85\textwidth}{!}
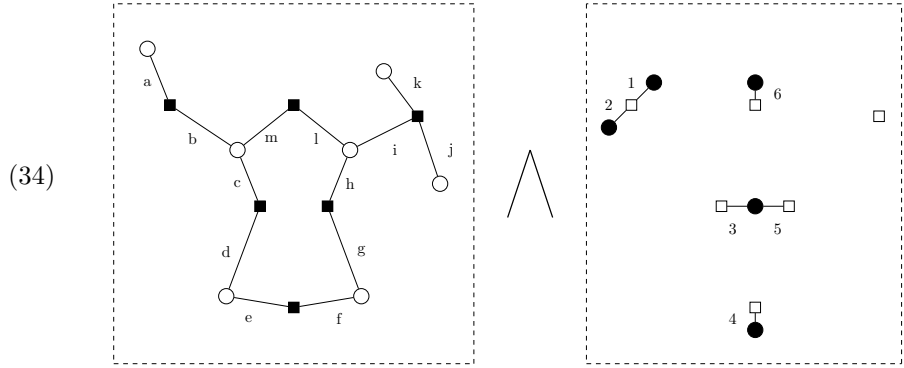
{
    \begin{tikzpicture}[baseline=(current  bounding  box.center)]
      \node [style=BlackSq] (0) at (-3.5, 2.25) {};
      \node [style=BlackSq] (1) at (-4.25, 0) {};
      \node [style=BlackSq] (2) at (-2.75, 0) {};
      \node [style=BlackSq] (3) at (-3.5, -2.25) {};
      \node [style=BlackSq] (4) at (-0.75, 2) {};
      \node [style=BlackSq] (5) at (-6.25, 2.25) {};
      \node [style=White] (7) at (-4.75, 1.25) {};
      \node [style=White] (8) at (-2.25, 1.25) {};
      \node [style=White] (9) at (-5, -2) {};
      \node [style=White] (10) at (-2, -2) {};
      \node [style=White] (11) at (-0.25, 0.5) {};
      \node [style=White] (12) at (-1.5, 3) {};
      \node [style=White] (14) at (-6.75, 3.5) {};
      \node [style=none] (15) at (-6.75, 2.75) {a};
      \node [style=none] (16) at (-5.75, 1.5) {b};
      \node [style=none] (17) at (-4.75, 0.5) {c};
      \node [style=none] (18) at (-5, -1) {d};
      \node [style=none] (19) at (-4.5, -2.5) {e};
      \node [style=none] (20) at (-2.5, -2.5) {f};
      \node [style=none] (21) at (-2, -1) {g};
      \node [style=none] (22) at (-2.25, 0.5) {h};
      \node [style=none] (23) at (-1.25, 1.25) {i};
      \node [style=none] (24) at (0, 1.25) {j};
      \node [style=none] (25) at (-0.75, 2.75) {k};
      \node [style=none] (26) at (-3, 1.5) {l};
      \node [style=none] (27) at (-4, 1.5) {m};
      \node [style=WhiteSq] (31) at (6.75, 2.25) {};
      \node [style=WhiteSq] (32) at (6, 0) {};
      \node [style=WhiteSq] (33) at (7.5, 0) {};
      \node [style=WhiteSq] (34) at (6.75, -2.25) {};
      \node [style=WhiteSq] (35) at (9.5, 2) {};
      \node [style=WhiteSq] (36) at (4, 2.25) {};
      \node [style=Black] (38) at (6.75, 0) {};
      \node [style=Black] (39) at (4.5, 2.75) {};
      \node [style=Black] (40) at (3.5, 1.75) {};
      \node [style=Black] (42) at (6.75, 2.75) {};
      \node [style=Black] (43) at (6.75, -2.75) {};
      \node [style=none] (44) at (6.25, -2.5) {4};
      \node [style=none] (45) at (3.5, 2.25) {2};
      \node [style=none] (46) at (4, 2.75) {1};
      \node [style=none] (47) at (7.25, 2.5) {6};
      \node [style=none] (48) at (6.25, -0.5) {3};
      \node [style=none] (49) at (7.25, -0.5) {5};
      \node [style=none] (50) at (-7.5, 4.5) {};
      \node [style=none] (51) at (0.5, 4.5) {};
      \node [style=none] (52) at (0.5, -3.5) {};
      \node [style=none] (53) at (-7.5, -3.5) {};
      \node [style=none] (54) at (3, 4.5) {};
      \node [style=none] (55) at (3, -3.5) {};
      \node [style=none] (56) at (10, 4.5) {};
      \node [style=none] (57) at (10, -3.5) {};
      \node [style=none] (58) at (1.25, -0.25) {};
      \node [style=none] (59) at (2.25, -0.25) {};
      \node [style=none] (60) at (1.75, 1.25) {};
      \draw (7) to (1);
      \draw (1) to (9);
      \draw (9) to (3);
      \draw (3) to (10);
      \draw (10) to (2);
      \draw (2) to (8);
      \draw (8) to (0);
      \draw (0) to (7);
      \draw (5) to (7);
      \draw (14) to (5);
      \draw (8) to (4);
      \draw (12) to (4);
      \draw (4) to (11);
      \draw (32) to (38);
      \draw (38) to (33);
      \draw (39) to (36);
      \draw (36) to (40);
      \draw (42) to (31);
      \draw (34) to (43);
      \draw [dashed] (50.center) to (53.center);
      \draw [dashed] (53.center) to (52.center);
      \draw [dashed] (52.center) to (51.center);
      \draw [dashed] (51.center) to (50.center);
      \draw [dashed] (54.center) to (55.center);
      \draw [dashed] (55.center) to (57.center);
      \draw [dashed] (57.center) to (56.center);
      \draw [dashed] (56.center) to (54.center);
      \draw [thick] (60.center) to (58.center);
      \draw [thick] (60.center) to (59.center);
    \end{tikzpicture}
    }
  \end{equation}

  To evaluate this, one can imagine superimposing the two graphs so that the squares are matched and then connecting the black and white vertices accordingly to obtain \eqref{orion-result}.
  The elbow of the arm holding the club gets ``cloned'', Orion's belt gets tightened, and the black vertex that is part of the shield is removed.
  \begin{equation}
    \label{orion-result}
    \resizebox{0.40\textwidth}{!}
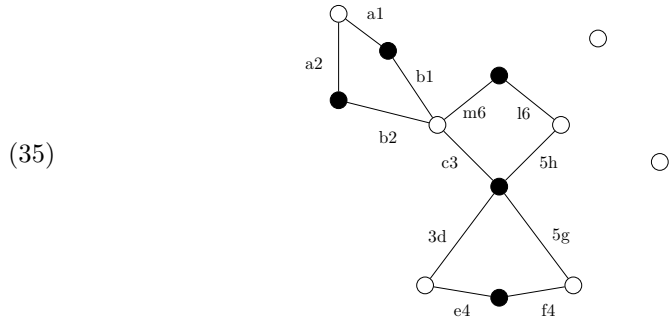
{
    \begin{tikzpicture}[baseline=(current  bounding  box.center)]
      \node [style=White] (57) at (-1.25, 1.25) {};
      \node [style=White] (58) at (1.25, 1.25) {};
      \node [style=White] (59) at (-1.5, -2) {};
      \node [style=White] (60) at (1.5, -2) {};
      \node [style=White] (61) at (3.25, 0.5) {};
      \node [style=White] (62) at (2, 3) {};
      \node [style=White] (63) at (-3.25, 3.5) {};
      \node [style=Black] (83) at (0, 0) {};
      \node [style=Black] (84) at (-2.25, 2.75) {};
      \node [style=Black] (85) at (-3.25, 1.75) {};
      \node [style=Black] (86) at (0, 2.25) {};
      \node [style=Black] (87) at (0, -2.25) {};
      \node [style=none] (88) at (-3.75, 2.5) {a2};
      \node [style=none] (89) at (-2.25, 1) {b2};
      \node [style=none] (90) at (-1, 0.5) {c3};
      \node [style=none] (91) at (-1.25, -1) {3d};
      \node [style=none] (92) at (-0.75, -2.5) {e4};
      \node [style=none] (93) at (1, -2.5) {f4};
      \node [style=none] (94) at (1.25, -1) {5g};
      \node [style=none] (95) at (1, 0.5) {5h};
      \node [style=none] (96) at (0.5, 1.5) {l6};
      \node [style=none] (97) at (-0.5, 1.5) {m6};
      \node [style=none] (98) at (-1.5, 2.25) {b1};
      \node [style=none] (99) at (-2.5, 3.5) {a1};
      \draw (63) to (85);
      \draw (85) to (57);
      \draw (63) to (84);
      \draw (84) to (57);
      \draw (57) to (86);
      \draw (86) to (58);
      \draw (57) to (83);
      \draw (58) to (83);
      \draw (83) to (59);
      \draw (59) to (87);
      \draw (87) to (60);
      \draw (60) to (83);
    \end{tikzpicture}
    }
  \end{equation}
\end{ex}

\subsection{Bialgebras}
\label{sec:bialgebra}

The two middle arrows in a composition of two spans forms a cospan.
We know that cospans factor as $\{V_y \rightarrow \{y\} \leftarrow V'_y\}_{y \in Y}$.
Hence it suffices to look at these basic cospans to understand the composition of spans.
For instance, the pullback of $\{a,b,c\} \rightarrow pt \leftarrow \{1, 2\}$ is the span $\{a,b,c\} \leftarrow \{a,b,c\} \times \{1,2\} \rightarrow \{1, 2\}$.
Writing out the elements gives the following picture:
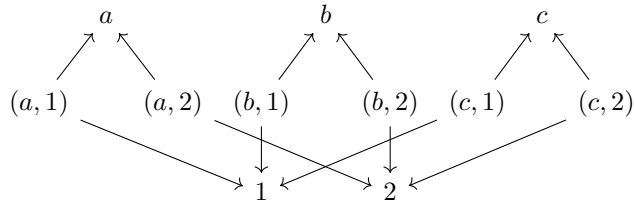
\begin{equation*}
  \begin{tikzcd}[column sep = 0.1cm]
    & a &&& b &&& c \\
    {(a,1)} && {(a,2)} & {(b,1)} && {(b,2)} & {(c,1)} && {(c,2)} \\
    &&& 1 && 2
    \arrow[from=2-1, to=1-2]
    \arrow[from=2-3, to=1-2]
    \arrow[from=2-4, to=1-5]
    \arrow[from=2-6, to=1-5]
    \arrow[from=2-1, to=3-4]
    \arrow[from=2-3, to=3-6]
    \arrow[from=2-4, to=3-4]
    \arrow[from=2-6, to=3-6]
    \arrow[from=2-7, to=1-8]
    \arrow[from=2-9, to=1-8]
    \arrow[from=2-7, to=3-4]
    \arrow[from=2-9, to=3-6]
  \end{tikzcd}
\end{equation*}
One can think of this picture as the complete graph between vertices $\{a, b, c\}$ and vertices $\{1, 2\}$ with each pair connected by an ``edge'' in $\{a, b, c\} \times \{1, 2\}$ exactly once.

This complete graph picture is a natural way to describe how multiplication and comultiplication interact in a bialgebra.
For example, the relation $\Delta \circ \mu = (\mu \otimes \mu) \circ (\id \otimes C \otimes \id) \circ (\Delta \otimes \Delta)$ for bialgebras can be depicted graphically by Diagram~\eqref{eq:alg-coalg-compat} where $\mu$ is represented as a circle and $\Delta$ is represented as a square.
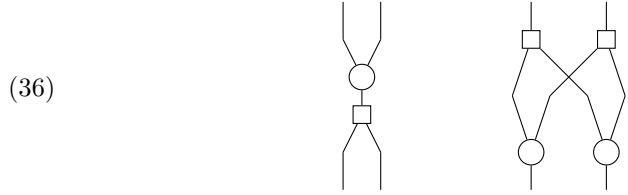
\begin{equation}
  \label{eq:alg-coalg-compat}
  \begin{tikzpicture}[baseline=(current  bounding  box.center)]
    \node [style=WhiteSq] (0) at (-2, 0.25) {};
    \node [style=White] (1) at (-2, 0.75) {};
    \node [style=none] (2) at (-2.25, 1.25) {};
    \node [style=none] (3) at (-1.75, 1.25) {};
    \node [style=none] (4) at (-2.25, -0.25) {};
    \node [style=none] (5) at (-1.75, -0.25) {};
    \node [style=none] (6) at (-2.25, 1.75) {};
    \node [style=none] (7) at (-1.75, 1.75) {};
    \node [style=none] (8) at (-2.25, -0.75) {};
    \node [style=none] (9) at (-1.75, -0.75) {};
    \node [style=WhiteSq] (12) at (0.25, 1.25) {};
    \node [style=WhiteSq] (13) at (1.25, 1.25) {};
    \node [style=White] (14) at (0.25, -0.25) {};
    \node [style=White] (15) at (1.25, -0.25) {};
    \node [style=none] (16) at (0.25, 1.75) {};
    \node [style=none] (17) at (1.25, 1.75) {};
    \node [style=none] (18) at (0.25, -0.75) {};
    \node [style=none] (19) at (1.25, -0.75) {};
    \node [style=none] (20) at (0, 0.5) {};
    \node [style=none] (21) at (0.5, 0.5) {};
    \node [style=none] (22) at (1, 0.5) {};
    \node [style=none] (23) at (1.5, 0.5) {};
    \draw (1) to (0);
    \draw (2.center) to (1);
    \draw (3.center) to (1);
    \draw (0) to (4.center);
    \draw (0) to (5.center);
    \draw (6.center) to (2.center);
    \draw (7.center) to (3.center);
    \draw (4.center) to (8.center);
    \draw (5.center) to (9.center);
    \draw (16.center) to (12);
    \draw (17.center) to (13);
    \draw (14) to (18.center);
    \draw (15) to (19.center);
    \draw (12) to (22.center);
    \draw (13) to (21.center);
    \draw (12) to (20.center);
    \draw (13) to (23.center);
    \draw (23.center) to (15);
    \draw (22.center) to (15);
    \draw (21.center) to (14);
    \draw (20.center) to (14);
  \end{tikzpicture}
\end{equation}
Conversely, we can recover the complete graph picture from this relation by using associativity:
\begin{equation}
  \resizebox{0.85\textwidth}{!}{
    \begin{tikzpicture}[baseline=(current  bounding  box.center)]
      \node [style=none] (0) at (-5.25, 0.5) {};
      \node [style=none] (1) at (-4.25, 0.5) {};
      \node [style=none] (2) at (-5.25, -0.5) {};
      \node [style=none] (3) at (-4.25, -0.5) {};
      \node [style=WhiteSq] (6) at (-3, 0.75) {};
      \node [style=WhiteSq] (7) at (-2, 0.5) {};
      \node [style=White] (8) at (-3, -0.5) {};
      \node [style=White] (9) at (-2, -0.5) {};
      \node [style=none] (10) at (-2.75, 0) {};
      \node [style=none] (11) at (-2.25, 0) {};
      \node [style=none] (12) at (-3.25, 0) {};
      \node [style=none] (13) at (-1.75, 0) {};
      \node [style=White] (14) at (-4.25, 1) {};
      \node [style=none] (15) at (-4.5, 1.5) {};
      \node [style=none] (16) at (-4, 1.5) {};
      \node [style=White] (17) at (-2, 1) {};
      \node [style=none] (18) at (-2.25, 1.5) {};
      \node [style=none] (19) at (-1.75, 1.5) {};
      \node [style=White] (20) at (-4.75, 0.25) {};
      \node [style=WhiteSq] (21) at (-4.75, -0.25) {};
      \node [style=WhiteSq] (22) at (-0.75, 0.75) {};
      \node [style=White] (24) at (-0.75, -0.5) {};
      \node [style=White] (25) at (0.75, -0.5) {};
      \node [style=White] (26) at (-0.5, 0) {};
      \node [style=none] (27) at (0.5, 0) {};
      \node [style=none] (28) at (-1, 0) {};
      \node [style=White] (29) at (1, 0) {};
      \node [style=WhiteSq] (31) at (0.25, 1.5) {};
      \node [style=WhiteSq] (32) at (0.75, 1.5) {};
      \node [style=none] (33) at (-0.25, 0.75) {};
      \node [style=none] (34) at (0.25, 0.75) {};
      \node [style=none] (35) at (0.75, 0.75) {};
      \node [style=none] (36) at (1.25, 0.75) {};
      \node [style=none] (37) at (-5.25, 2) {};
      \node [style=none] (38) at (-3, 2) {};
      \node [style=none] (39) at (-0.75, 2) {};
      \node [style=none] (40) at (-4.5, 2) {};
      \node [style=none] (41) at (-4, 2) {};
      \node [style=none] (42) at (-2.25, 2) {};
      \node [style=none] (43) at (-1.75, 2) {};
      \node [style=none] (44) at (0.25, 2) {};
      \node [style=none] (45) at (0.75, 2) {};
      \node [style=none] (46) at (0.75, -1) {};
      \node [style=none] (47) at (-0.75, -1) {};
      \node [style=none] (48) at (-2, -1) {};
      \node [style=none] (49) at (-3, -1) {};
      \node [style=none] (50) at (-4.25, -1) {};
      \node [style=none] (51) at (-5.25, -1) {};
      \node [style=WhiteSq] (52) at (2.5, 1.5) {};
      \node [style=White] (53) at (2.5, -0.5) {};
      \node [style=White] (54) at (4, -0.5) {};
      \node [style=none] (57) at (2, 0.5) {};
      \node [style=WhiteSq] (59) at (3.25, 1.5) {};
      \node [style=WhiteSq] (60) at (4, 1.5) {};
      \node [style=none] (61) at (2.5, 0.5) {};
      \node [style=none] (62) at (3, 0.5) {};
      \node [style=none] (63) at (4, 0.5) {};
      \node [style=none] (64) at (4.5, 0.5) {};
      \node [style=none] (65) at (2.5, 2) {};
      \node [style=none] (66) at (3.25, 2) {};
      \node [style=none] (67) at (4, 2) {};
      \node [style=none] (68) at (4, -1) {};
      \node [style=none] (69) at (2.5, -1) {};
      \node [style=none] (70) at (3.5, 0.5) {};
      \node [style=none] (73) at (-7.25, -0.5) {};
      \node [style=none] (74) at (-6.25, -0.5) {};
      \node [style=White] (78) at (-6.75, 0.25) {};
      \node [style=WhiteSq] (79) at (-6.75, -0.25) {};
      \node [style=none] (80) at (-7.25, 2) {};
      \node [style=none] (81) at (-6.75, 2) {};
      \node [style=none] (82) at (-6.25, 2) {};
      \node [style=none] (83) at (-6.25, -1) {};
      \node [style=none] (84) at (-7.25, -1) {};
      \node [style=none] (85) at (-7.25, 1) {};
      \node [style=none] (86) at (-6.75, 1) {};
      \node [style=none] (87) at (-6.25, 1) {};
      \draw (7) to (10.center);
      \draw (6) to (11.center);
      \draw (7) to (13.center);
      \draw (6) to (12.center);
      \draw (12.center) to (8);
      \draw (10.center) to (8);
      \draw (11.center) to (9);
      \draw (13.center) to (9);
      \draw (17) to (7);
      \draw (18.center) to (17);
      \draw (19.center) to (17);
      \draw (15.center) to (14);
      \draw (16.center) to (14);
      \draw (14) to (1.center);
      \draw (1.center) to (20);
      \draw (0.center) to (20);
      \draw (20) to (21);
      \draw (21) to (2.center);
      \draw (21) to (3.center);
      \draw (22) to (27.center);
      \draw (22) to (28.center);
      \draw (28.center) to (24);
      \draw (26) to (24);
      \draw (27.center) to (25);
      \draw (29) to (25);
      \draw (31) to (35.center);
      \draw (32) to (34.center);
      \draw (31) to (33.center);
      \draw (32) to (36.center);
      \draw (36.center) to (29);
      \draw (35.center) to (29);
      \draw (33.center) to (26);
      \draw (26) to (34.center);
      \draw (3.center) to (50.center);
      \draw (2.center) to (51.center);
      \draw (37.center) to (0.center);
      \draw (40.center) to (15.center);
      \draw (41.center) to (16.center);
      \draw (38.center) to (6);
      \draw (42.center) to (18.center);
      \draw (43.center) to (19.center);
      \draw (8) to (49.center);
      \draw (9) to (48.center);
      \draw (24) to (47.center);
      \draw (25) to (46.center);
      \draw (39.center) to (22);
      \draw (44.center) to (31);
      \draw (45.center) to (32);
      \draw (52) to (57.center);
      \draw (57.center) to (53);
      \draw (59) to (63.center);
      \draw (60) to (62.center);
      \draw (59) to (61.center);
      \draw (60) to (64.center);
      \draw (53) to (69.center);
      \draw (54) to (68.center);
      \draw (65.center) to (52);
      \draw (66.center) to (59);
      \draw (67.center) to (60);
      \draw (61.center) to (53);
      \draw (62.center) to (53);
      \draw (63.center) to (54);
      \draw (64.center) to (54);
      \draw (70.center) to (54);
      \draw (52) to (70.center);
      \draw (78) to (79);
      \draw (79) to (73.center);
      \draw (79) to (74.center);
      \draw (74.center) to (83.center);
      \draw (73.center) to (84.center);
      \draw (80.center) to (85.center);
      \draw (85.center) to (78);
      \draw (81.center) to (86.center);
      \draw (86.center) to (78);
      \draw (82.center) to (87.center);
      \draw (87.center) to (78);
    \end{tikzpicture}
  }
\end{equation}

\begin{prop}[Lack~\cite{lack2004composing}]
  \label{span-corep-bialg}
  $\Span$ corepresents commutative bimonoids, that is there is a one-to-one correspondence between bicommutative bimonoids in $\mathcal{C}$ and strong monoidal functors $M: \Span \to \mathcal{C}$.
\end{prop}

\begin{example}
  Let $A$ be a commutative unital monoid in $\Set$.
  For a finite set $I$, let $T_A(I)$ denote the set of maps $I \to A$.
  We will think of $T_A(I)$ as the set of $I$-tuples of $A$.
  Define a strong monoidal functor $M_A: \Span \to \Set$ by sending the span $X \overset{l}{\leftarrow} V \overset{r}{\rightarrow} Y$ to the map
  \begin{equation}
    T_A(X) \overset{l^*}{\to} T_A(V) \overset{\mu_r}{\to} T_A(Y)
  \end{equation}
  where $\mu_r$ and $l^*$ are defined below:
  \begin{description}
  \item[Product and unit] Given a map $I \overset{r}{\to} J$ of finite sets, define a map $\mu_r: T_A(I) \to T_A(J)$ so that the tuple $t \in T_A(I)$ gets sent to tuple $t' \in T_A(J)$ defined by:
    \begin{equation}
      t'(j) = \sum_{i \in r^{-1}(j)} t(i)
 \end{equation}
    We will also use the convention that the empty sum is the unit of $A$.
    This is well-defined because $A$ is commutative and $I$ is finite.
  \item[Coproduct and counit] Given a map $I \overset{l}{\leftarrow} J$, define $l^*: T_A(I) \to T_A(J)$ by precomposition.
    In words, the coproduct is duplication and the counit is deletion.
  \end{description}
\end{example}

\subsection{Plus construction of Span}
\label{sec:span-plus}

Building on the combinatorics described in Section~\ref{sec:black-white-graphs}, we will draw the objects of $\Span^+$ as black and white graphs with the white vertices on top and the black vertices on the bottom:
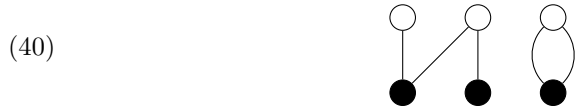
\begin{equation}
  \begin{tikzpicture}[baseline=(current bounding  box.center)]
    \node [style=White] (0) at (-1.5, 0.5) {};
    \node [style=Black] (1) at (-1.5, -0.5) {};
    \node [style=Black] (2) at (-0.5, -0.5) {};
    \node [style=Black] (3) at (0.5, -0.5) {};
    \node [style=White] (4) at (-0.5, 0.5) {};
    \node [style=White] (5) at (0.5, 0.5) {};
    \draw (0) to (1);
    \draw (4) to (1);
    \draw (4) to (2);
    \draw [bend right=45] (5) to (3);
    \draw [bend left=45] (5) to (3);
  \end{tikzpicture}
\end{equation}

Under this interpretation, the $\gamma$-morphisms of the plus construction can be understood as a process which matches black and white vertices and then resolves the connections according to the rule described in Section~\ref{sec:black-white-graphs}:
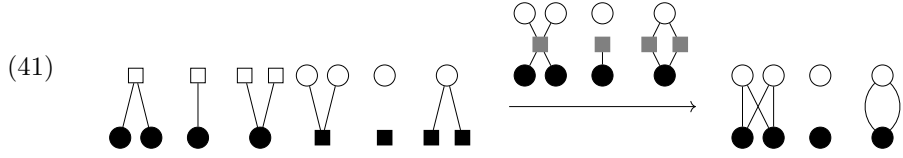
\begin{equation}
  \resizebox{0.85\textwidth}{!}{
    \begin{tikzpicture}[baseline=(current bounding  box.center)]
      \node [style=WhiteSq] (0) at (-6.5, 0.5) {};
      \node [style=WhiteSq] (1) at (-5.5, 0.5) {};
      \node [style=WhiteSq] (2) at (-4.75, 0.5) {};
      \node [style=WhiteSq] (3) at (-4.25, 0.5) {};
      \node [style=Black] (4) at (-6.75, -0.5) {};
      \node [style=Black] (5) at (-6.25, -0.5) {};
      \node [style=Black] (6) at (-5.5, -0.5) {};
      \node [style=Black] (7) at (-4.5, -0.5) {};
      \node [style=White] (8) at (-3.75, 0.5) {};
      \node [style=White] (9) at (-3.25, 0.5) {};
      \node [style=White] (10) at (-2.5, 0.5) {};
      \node [style=White] (11) at (-1.5, 0.5) {};
      \node [style=BlackSq] (12) at (-3.5, -0.5) {};
      \node [style=BlackSq] (13) at (-2.5, -0.5) {};
      \node [style=BlackSq] (14) at (-1.75, -0.5) {};
      \node [style=BlackSq] (15) at (-1.25, -0.5) {};
      \node [style=Black] (20) at (-0.25, 0.5) {};
      \node [style=Black] (21) at (0.25, 0.5) {};
      \node [style=Black] (22) at (1, 0.5) {};
      \node [style=Black] (23) at (2, 0.5) {};
      \node [style=White] (24) at (-0.25, 1.5) {};
      \node [style=White] (25) at (0.25, 1.5) {};
      \node [style=White] (26) at (1, 1.5) {};
      \node [style=White] (27) at (2, 1.5) {};
      \node [style=WhiteSq, gray] (32) at (0, 1) {};
      \node [style=WhiteSq, gray] (33) at (1, 1) {};
      \node [style=WhiteSq, gray] (34) at (1.75, 1) {};
      \node [style=WhiteSq, gray] (35) at (2.25, 1) {};
      \node [style=Black] (36) at (3.25, -0.5) {};
      \node [style=Black] (37) at (3.75, -0.5) {};
      \node [style=Black] (38) at (4.5, -0.5) {};
      \node [style=Black] (39) at (5.5, -0.5) {};
      \node [style=White] (40) at (3.25, 0.5) {};
      \node [style=White] (41) at (3.75, 0.5) {};
      \node [style=White] (42) at (4.5, 0.5) {};
      \node [style=White] (43) at (5.5, 0.5) {};
      \node [style=none] (44) at (-0.5, 0) {};
      \node [style=none] (45) at (2.5, 0) {};
      \draw (0) to (4);
      \draw (0) to (5);
      \draw (1) to (6);
      \draw (2) to (7);
      \draw (3) to (7);
      \draw (8) to (12);
      \draw (9) to (12);
      \draw (11) to (14);
      \draw (11) to (15);
      \draw (24) to (32);
      \draw (25) to (32);
      \draw (27) to (34);
      \draw (27) to (35);
      \draw (20) to (32);
      \draw (21) to (32);
      \draw (22) to (33);
      \draw (23) to (34);
      \draw (23) to (35);
      \draw [bend right=45] (43) to (39);
      \draw [bend left=45] (43) to (39);
      \draw (40) to (37);
      \draw (41) to (36);
      \draw (41) to (37);
      \draw (40) to (36);
      \draw [style=Arrow] (44.center) to (45.center);
    \end{tikzpicture}
  }
\end{equation}

In a rough sense, this picture is dual to the one we described for $\Cospan^+$ in Proposition~\ref{cospan-plus-as-graph}.
For $\Cospan^+$, flags were grouped by vertices.
The $\gamma$-morphisms combine multiple groupings into a single grouping of flags.
For $\Span^+$, vertices are connected by edges.
The $\gamma$-morphisms resolve a series of connections into a single connection.

To understand the algebraic significance of $\Span^+$, consider some of the results for the other combinatorial categories:
\begin{description}
\item[Trivial Feynman category]
  Strong monoidal functors $\FF^{triv} \to \mathcal{C}$ are objects in $\mathcal{C}$.
  Strong monoidal functors $(\FF^{triv})^+ \to \mathcal{C}$ are monoids in $\mathcal{C}$.
  (Proposition~\ref{triv-plus-corep-monoids})
\item[Finite sets]
  Strong monoidal functors $\FinSet \to \mathcal{C}$ are commutative algebras in $\mathcal{C}$.
  (Proposition~\ref{FinSet-corep-com-monoids})
  Strong monoidal functors $\FinSet^+ \to \mathcal{C}$ are operads in $\mathcal{C}$.
  (Proposition~\ref{FinSet-plus-corep-operads})
\item[Cospans]
  Strong monoidal functors $\Cospan \to \mathcal{C}$ are special bicommutative Frobenius algebras in $\mathcal{C}$.
  (Section~\ref{sec:cospan-corep-frobenius})
  Strong monoidal functors $\Cospan^+ \to \mathcal{C}$ are properads in $\mathcal{C}$.
  (Proposition~\ref{cospan-plus-as-graph})
\end{description}
Strong monoidal functors $\Span \to \mathcal{C}$ are bicommutative bialgebras (Section~\ref{sec:bialgebra}).
Hence, by analogy, a strong monoidal functor $\Span^+ \to \mathcal{C}$ encodes a kind of bialgebra just like operads encode a kind of algebra.

\begin{example}[Associative $\Span^+$-op]
  Define a strong monoidal functor $A: Span^+ \to \Set$ so that $A(X \overset{l}{\leftarrow} R \overset{r}{\rightarrow} Y) = OF(l) \times OF(r)$ where $OF(f)$ is the set of orders on the fibers of a function $f$.
  This datum composes in the canonical fashion.
  For example, suppose we have ordered fibers $\{a < b < c\} \rightarrow pt \leftarrow \{1 < 2\}$, then the composition would produce the following ordered fibers:
  \begin{center}
    \begin{tabular}{rl}
      $\{a\} \leftarrow \{(a,1) < (a,2)\}$
      & \\
      & $\{(a,1) < (b,1) < (c,1)\} \rightarrow \{1\}$ \\
      $\{b\} \leftarrow \{(b,1) < (b,2)\}$
      & \\
      & $\{(a,2) < (b,2) < (c,2)\} \rightarrow \{2\}$ \\
      $\{c\} \leftarrow \{(c,1) < (c,2)\}$ &
    \end{tabular}
  \end{center}

  If we think of an object $x$ of $\Span^+$ as a black and white graph, then an element of $A(x)$ is an ordering of the edges at each vertex of $x$:
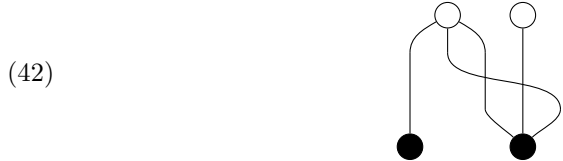
  \begin{equation}
    \begin{tikzpicture}[baseline=(current bounding  box.center)]
      \node [style=White] (0) at (-0.5, 1) {};
      \node [style=White] (1) at (0.5, 1) {};
      \node [style=none] (2) at (-0.5, 0.5) {};
      \node [style=none] (3) at (0, 0.5) {};
      \node [style=none] (4) at (0, -0.25) {};
      \node [style=none] (5) at (1, -0.25) {};
      \node [style=Black] (6) at (-1, -0.75) {};
      \node [style=Black] (7) at (0.5, -0.75) {};
      \node [style=none] (8) at (-1, 0.5) {};
      \draw [in=90, out=-90, looseness=0.75] (2.center) to (5.center);
      \draw [in=90, out=-90, looseness=0.75] (3.center) to (4.center);
      \draw [in=135, out=-90, looseness=0.50] (4.center) to (7);
      \draw [in=45, out=-90, looseness=0.75] (5.center) to (7);
      \draw (0) to (2.center);
      \draw [in=90, out=-30] (0) to (3.center);
      \draw (1) to (7);
      \draw [in=90, out=-150] (0) to (8.center);
      \draw (8.center) to (6);
    \end{tikzpicture}
  \end{equation}
  The $\gamma$-morphisms compose this data so that the order of the black and white vertices incorporates the ordering of the intermediate vertices.
  
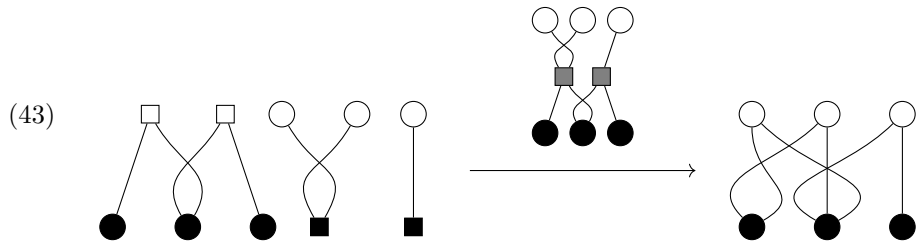
\begin{equation}
    \begin{tikzpicture}[baseline=(current bounding  box.center)]
      \node [style=White] (0) at (-0.5, 2) {};
      \node [style=GraySq] (1) at (-0.25, 1.25) {};
      \node [style=White] (2) at (0, 2) {};
      \node [style=White] (3) at (0.5, 2) {};
      \node [style=Black] (4) at (-0.5, 0.5) {};
      \node [style=Black] (5) at (0, 0.5) {};
      \node [style=Black] (6) at (0.5, 0.5) {};
      \node [style=GraySq] (7) at (0.25, 1.25) {};
      \node [style=White] (8) at (-4, 0.75) {};
      \node [style=BlackSq] (9) at (-3.5, -0.75) {};
      \node [style=White] (10) at (-3, 0.75) {};
      \node [style=White] (11) at (-2.25, 0.75) {};
      \node [style=BlackSq] (12) at (-2.25, -0.75) {};
      \node [style=WhiteSq] (13) at (-5.75, 0.75) {};
      \node [style=Black] (14) at (-6.25, -0.75) {};
      \node [style=Black] (15) at (-5.25, -0.75) {};
      \node [style=Black] (16) at (-4.25, -0.75) {};
      \node [style=WhiteSq] (17) at (-4.75, 0.75) {};
      \node [style=White] (18) at (2.25, 0.75) {};
      \node [style=White] (20) at (3.25, 0.75) {};
      \node [style=White] (21) at (4.25, 0.75) {};
      \node [style=Black] (22) at (2.25, -0.75) {};
      \node [style=Black] (23) at (3.25, -0.75) {};
      \node [style=Black] (24) at (4.25, -0.75) {};
      \node [style=none] (25) at (-1.5, 0) {};
      \node [style=none] (26) at (1.5, 0) {};
      \draw [in=60, out=-60, looseness=1.50] (0) to (1);
      \draw [in=120, out=-120, looseness=1.50] (2) to (1);
      \draw (1) to (4);
      \draw [in=60, out=-60, looseness=1.25] (1) to (5);
      \draw [in=120, out=-120, looseness=1.25] (7) to (5);
      \draw (3) to (7);
      \draw (7) to (6);
      \draw [in=60, out=-60, looseness=1.50] (8) to (9);
      \draw [in=120, out=-120, looseness=1.50] (10) to (9);
      \draw (11) to (12);
      \draw (13) to (14);
      \draw [in=60, out=-60, looseness=1.25] (13) to (15);
      \draw [in=120, out=-120, looseness=1.25] (17) to (15);
      \draw (17) to (16);
      \draw [in=30, out=-90, looseness=1.50] (18) to (22);
      \draw [in=135, out=-135, looseness=1.25] (20) to (22);
      \draw [in=30, out=-45, looseness=1.50] (18) to (23);
      \draw (20) to (23);
      \draw [in=225, out=150, looseness=1.50] (23) to (21);
      \draw (21) to (24);
      \draw [style=Arrow] (25.center) to (26.center);
    \end{tikzpicture}
  \end{equation}
\end{example}

\begin{prop}
 The indexed enrichment $\Span_{A}$ corepresents associative bimonoids.
\end{prop}
\begin{proof}
  The only difference between $\Span_{A}$ and $\Span$ is that permuting fibers changes the morphisms of $\Span_{A}$ but keeps the morphisms of $\Span$ the same.
  Hence a strong monoidal functor $B: \Span_A \to \mathcal{C}$ is the same data as a bimonoid which is not necessarily commutative or cocommutative.
\end{proof}

\bibliography{Plus-Computation}
\bibliographystyle{amsalpha}

\end{document}